\documentclass[11pt]{amsart}

\usepackage[T1]{fontenc}
\usepackage[utf8]{inputenc}
\usepackage{amsmath,amssymb,amsthm,mathrsfs}
\usepackage{enumitem}
\usepackage[a4paper,margin=0.9in]{geometry}
\usepackage[colorlinks=true,linkcolor=blue,citecolor=blue,urlcolor=blue]{hyperref}

\numberwithin{equation}{section}

\newtheorem{theorem}{Theorem}[section]
\newtheorem{lemma}[theorem]{Lemma}
\newtheorem{proposition}[theorem]{Proposition}

\newtheorem{assumption}{Assumption}[section]
\theoremstyle{definition}
\newtheorem{definition}{Definition}[section]
\theoremstyle{remark}
\newtheorem{remark}{Remark}[section]

\newcommand{\R}{\mathbb R}
\newcommand{\E}{\mathbb E}
\newcommand{\PP}{\mathbb P}
\newcommand{\cF}{\mathcal F}

\newcommand{\elltwo}{\ell_2}
\newcommand{\md}{\mathop{}\!\mathrm{d}}

\title[Dirichlet problem for stochastic parabolic equations]{Schauder estimates and classical solutions of the Dirichlet problem for stochastic parabolic equations}

\author{Kai Du}
\address{Shanghai Center for Mathematical Sciences, Fudan University, Shanghai 200438, China.}
\email{kdu@fudan.edu.cn}
\date{\today}


\begin{document}

\begin{abstract}
We study second-order stochastic parabolic equations in a cylindrical domain with homogeneous Dirichlet boundary conditions. Under a natural compatibility condition on the gradient-type noise, we establish global Schauder estimates in stochastic H\"older spaces for the Dirichlet problem. The coefficients and free terms are assumed to be H\"older continuous in the spatial variables, while time H\"older continuity is imposed only on certain boundary traces. As a consequence, we obtain existence and uniqueness of quasi-classical solutions in stochastic H\"older spaces, and further derive pathwise classical solvability in H\"older classes on bounded domains.
\end{abstract}

\maketitle

\section{Introduction}

Higher regularity theory for nonlinear stochastic partial differential equations is often based on the corresponding properties of the associated linear equations. In this paper we consider the Dirichlet problem for stochastic parabolic equations of It\^o type in a domain $G\subset \R^n$:
\begin{equation*}\label{eq:intro-main}
du = (a^{ij}u_{x_i x_j}+b^iu_{x_i}+cu+f)\md t + (\sigma^{ik}u_{x_i}+\nu^k u+g^k)\md w_t^k
\end{equation*}
under homogeneous boundary conditions. Here $\{w^k\}_{k\ge1}$ are independent standard Wiener processes defined on a filtered complete probability space, and the coefficients, the free terms, and the unknown function are all adapted random fields. Our purpose is to develop a Schauder-type theory for this problem under suitable compatibility conditions and, as a consequence, to obtain classical solvability.

The regularity theory of parabolic SPDEs has been studied for several decades. In the framework of Sobolev spaces, the well-posedness theory in variational form goes back to the work of Pardoux and Krylov--Rozovski\u{\i}; see~\cite{Pardoux1975,KrylovRozovskii1977}. Later, a rather complete $L^p$-theory for the Cauchy problem in the whole space was established by Krylov~\cite{Krylov1996Lp,Krylov1999Analytic}, and important extensions to boundary value problems were obtained by Krylov, Lototsky, Kim, and others; see, for example,~\cite{Krylov1994Dirichlet,KrylovLototsky1999,Lototsky2000,Kim2004Divergence,Kim2004Variable}. In particular, under suitable compatibility conditions one can recover the usual $W^{2,p}$-regularity in domains instead of working in weighted Sobolev spaces; see~\cite{Du2020}. Another important line of development is based on stochastic maximal regularity in Banach spaces; see~\cite{vanNeervenVeraarWeis2012,PortalVeraar2019} and the references therein. 

There is also a parallel line of research based on H\"older spaces, corresponding to the classical Schauder theory for deterministic elliptic and parabolic equations; see~\cite{GilbargTrudinger2001,Krylov1996Holder}. For the Cauchy problem in the whole space, the sharp $C^{2+\alpha}$-theory was obtained by Mikulevi\v{c}ius~\cite{Mikulevicius2000} for special classes of equations, and later in full generality by Du--Liu~\cite{DuLiu2019}. This gives a satisfactory H\"older theory in the absence of boundaries. The theory has been further extended through stochastic maximal regularity and singular-integral techniques~\cite{AgrestiVeraar2022,AgrestiVeraar2025,LoristVeraar2021}, and regularity estimates play a central role in numerical analysis and approximation schemes for SPDEs~\cite{KovacsLangPetersson2023,BreitProhl2024,BreitProhlWichman2025}.

For the Dirichlet problem, however, the situation is more delicate. A basic difficulty comes from the possible singular behavior of derivatives near the boundary. In contrast to deterministic equations, gradient noise may produce severe boundary irregularities even in simple one-dimensional examples. This was illustrated by Krylov~\cite{Krylov2003Brownian}, and shows that one cannot expect a boundary Schauder theory under assumptions comparable to those used in the whole-space case. In the literature, boundary value problems for SPDEs have been treated successfully in weighted Sobolev spaces and weighted H\"older spaces; see~\cite{Krylov1994Dirichlet,Lototsky2000,Kim2004Divergence,Kim2004Variable,MikuleviciusPragarauskas2003}. Such theories allow derivatives of solutions to grow near the boundary, and hence are well suited for equations with rough boundary behavior. If one is interested instead in bounded derivatives up to the boundary, then some compatibility conditions are needed.

A natural compatibility condition was proposed in our previous work~\cite{Du2020} on $W^{2,p}$-solutions in general domains: the vector fields $\sigma_{\cdot k}$ are tangent to the boundary, namely
\begin{equation*}\label{eq:intro-compatibility}
n(x)\cdot \sigma_{\cdot k}(t,x)=0 \qquad \text{on }\partial G
\end{equation*}
for all $k$, where $n(x)$ denotes the unit outward normal vector at $x\in\partial G$. Under this condition, together with the vanishing boundary trace of $g$, one can obtain $W^{2,p}$-regularity in the usual Sobolev spaces. This shows that the tangentiality condition is an appropriate geometric condition for recovering second-order boundary regularity in the Sobolev setting. The aim of the present paper is to show that, under suitable regularity assumptions, the same condition is also sufficient for a boundary Schauder theory.

More precisely, we establish global Schauder estimates in stochastic H\"older spaces for the Dirichlet problem under the above compatibility condition (Theorem~\ref{thm:schauder}). 
These estimates yield existence and uniqueness of quasi-classical solutions in stochastic H\"older spaces (Theorem~\ref{thm:solvability}). 
As a further consequence, we obtain pathwise classical solvability in H\"older classes on bounded domains (Theorem~\ref{thm:classical}). 
In this sense, the present work may be viewed as the boundary counterpart of the sharp H\"older theory for the Cauchy problem in the whole space developed in our previous work~\cite{DuLiu2019}. Together with the $W^{2,p}$-regularity results of~\cite{Du2020} established under the same tangentiality condition, these theorems indicate that this condition provides a unified geometric foundation for both Sobolev and Schauder boundary regularity theories for SPDEs on domains.

The proof proceeds by first flattening the boundary to a half-space and establishing estimates for a frozen-coefficient model equation. Tangential derivatives are handled by adapting the localization and approximation scheme of~\cite{DuLiu2019} to boundary cylinders.

For the first normal derivative, an additional difficulty arises from the fact that the second-order operator is not diagonal: a direct odd extension is not available, and stochastic flow transformations do not directly preserve the $L_\omega^\gamma$-valued H\"older norms used here. We overcome this by a reformulation of the equation that exploits the algebraic structure of the cross-derivative terms, allowing an odd/even extension and reducing the estimate to the whole-space theory of~\cite{DuLiu2019}.

The most difficult part is controlling the second normal derivative. After removing the stochastic noise, the reduced equation still carries a contribution from the original forcing on the flat boundary, which prevents a direct odd extension. This obstacle is removed by splitting off the boundary trace through a further correction term. The estimation of this term brings in the central novelty of the paper: a boundary estimate for the one-dimensional heat equation on the half-line with time-dependent Dirichlet data (Lemma~\ref{lem:half-line-boundary}). This lemma provides optimal $C^{\alpha,\alpha/2}$-control of the time derivative and the second spatial derivative in stochastic H\"older norms solely in terms of the H\"older seminorm of the time derivative of the boundary datum.

The paper is organized as follows. In Section~\ref{sec:main-results} we introduce the basic notation and state the main results. In Section~\ref{sec:halfspace} we study a model equation in a half-space and establish the boundary H\"older estimates needed later. In Section~\ref{sec:global-estimate} we pass from the half-space model to general domains by boundary flattening, perturbation, and localization, and complete the proof of Theorem~\ref{thm:schauder}. In Section~\ref{sec:solvability} we prove the solvability theorem in stochastic H\"older spaces by the method of continuity. Finally, in Section~\ref{sec:classical} we derive the pathwise classical solvability result.

\section{Main results}\label{sec:main-results}

Let $T>0$ be fixed, and let
\[
Q_T:=G\times(0,T], \qquad \bar Q_T:=\bar G\times[0,T], \qquad \Gamma_T:=\partial G\times[0,T].
\]
Throughout the paper, $(\Omega,\cF,(\cF_t)_{t\ge0},\PP)$ is a complete filtered probability space carrying a sequence of independent standard Wiener processes $\{w_t^k\}_{k\ge1}$. We consider the Dirichlet problem
\begin{equation}\label{eq:main}
du=\bigl(a^{ij}D_{ij}u+b^iD_i u+cu+f\bigr)\md t+\bigl(\sigma^{ik}D_i u+\nu^k u+g^k\bigr)\md w_t^k
\quad \text{in }Q_T,
\end{equation}
with the boundary and initial conditions
\begin{equation}\label{eq:boundary-initial}
u=0 \quad \text{on } \Gamma_T,
\qquad
u(0,\cdot)=0 \quad \text{in } G.
\end{equation}
Einstein's summation convention is used with $i,j=1,\dots,n$ and $k=1,2,\dots$ in this paper.
We take $a^{ij}=a^{ji}$ without loss of generality, since replacing $a$ by its symmetric part changes neither the differential operator nor the parabolicity condition below. For a possibly unbounded domain of class $C^{2+\alpha}$, we require boundary charts of a common positive radius with uniformly bounded $C^{2+\alpha}$ norms for the charts and their inverses; these bounds form the $C^{2+\alpha}$-character of $G$.

For a Banach space $E$, a domain $\mathcal O\subset\R^n$, and an interval $I\subset\R$, write $Q=\mathcal O\times I$. For a function $h:Q\to E$, an integer $m\ge 0$, and $\alpha\in(0,1)$, we set
\[
|h|_{m;Q}^E:=\max_{|\beta|\le m}\sup_{(x,t)\in Q}\|D^\beta h(x,t)\|_E,
\qquad
[h]_{m+\alpha;Q}^E:=\max_{|\beta|=m}\sup_{t\in I}\sup_{x\ne y\in\mathcal O}
\frac{\|D^\beta h(x,t)-D^\beta h(y,t)\|_E}{|x-y|^\alpha},
\]
and
\[
[h]_{(m+\alpha,\alpha/2);Q}^E:=\max_{|\beta|=m}
\sup_{\substack{(x,t),(y,s)\in Q\\(x,t)\ne(y,s)}}
\frac{\|D^\beta h(x,t)-D^\beta h(y,s)\|_E}{|x-y|^\alpha+|t-s|^{\alpha/2}}.
\]
We further write
\[
|h|_{m+\alpha;Q}^E:=|h|_{m;Q}^E+[h]_{m+\alpha;Q}^E,
\qquad
|h|_{(m+\alpha,\alpha/2);Q}^E:=|h|_{m;Q}^E+[h]_{(m+\alpha,\alpha/2);Q}^E.
\]
In this paper, $E$ will be either $L_\omega^\gamma:=L^\gamma(\Omega;\R)$ or $L_\omega^\gamma(\elltwo):=L^\gamma(\Omega;\elltwo)$, where $\gamma\ge2$ is fixed. The spaces $C_x^{m+\alpha}(Q;E)$ and $C_{x,t}^{m+\alpha,\alpha/2}(Q;E)$ are defined in the obvious way. For functions on $\Gamma_T$, the parabolic H\"older norm $|\cdot|_{(\alpha,\alpha/2);\Gamma_T}^E$ is understood in the standard sense.

\begin{definition}[Quasi-classical solution]
A predictable random field $u:\bar Q_T\times\Omega\to\R$ is called a quasi-classical solution of \eqref{eq:main}--\eqref{eq:boundary-initial} if
\begin{enumerate}[label=\rm(\roman*)]
\item for each $t\in(0,T]$, the map $x\mapsto u(x,t)$ is twice strongly differentiable from $G$ into $L_\omega^\gamma$, and $u(\cdot,t)$ together with its spatial derivatives up to order two extends continuously to $\bar G$ as an $L_\omega^\gamma$-valued function;
\item for each $x\in G$, the process $u(x,\cdot)$ is stochastically continuous and satisfies
\[
\begin{aligned}
u(x,t)-u(x,0) & =\int_0^t \bigl(a^{ij}D_{ij}u+b^iD_i u+cu+f\bigr)(x,s)\md s \\
& \quad +\int_0^t \bigl(\sigma^{ik}D_i u+\nu^k u+g^k\bigr)(x,s)\md w_s^k
\end{aligned}
\]
almost surely for all $t\in[0,T]$;
\item $u=0$ on $\Gamma_T$ and $u(0,\cdot)=0$ in $G$ almost surely.
\end{enumerate}
\end{definition}

The standing assumptions are as follows.

\begin{assumption}\label{ass:compatibility}
For each $k\ge 1$,
\begin{equation}\label{eq:compatibility}
n(x)\cdot \sigma_{\cdot k}(t,x,\omega)=0 \qquad \text{for all }(t,x)\in[0,T]\times\partial G
\end{equation}
almost surely, where $n(x)$ denotes the unit outward normal vector at $x\in\partial G$.
\end{assumption}

\begin{assumption}\label{ass:parabolicity}
There exist constants $\kappa,K>0$ such that
\begin{equation}\label{eq:parabolicity}
\kappa|\xi|^2+\sigma^{ik}(t,x)\sigma^{jk}(t,x)\xi_i\xi_j\le 2a^{ij}(t,x)\xi_i\xi_j\le K|\xi|^2
\end{equation}
for all $(t,x,\omega)\in[0,T]\times G\times\Omega$ and all $\xi\in\R^n$.
\end{assumption}

\begin{assumption}\label{ass:coefficients}
The coefficients
\[
a^{ij},\, b^i,\, c: \bar Q_T\times\Omega\to\R, \qquad \sigma^i=(\sigma^{ik})_{k\ge1},\, \nu=(\nu^k)_{k\ge1}: \bar Q_T\times\Omega\to\elltwo
\]
are predictable. Moreover, for some $\alpha\in(0,1)$,
\[
a^{ij},\ b^i,\ c\in C_x^\alpha(\bar Q_T;L_\omega^\infty),
\qquad
\sigma^i,\ \nu\in C_x^{1+\alpha}(\bar Q_T;L_\omega^\infty(\elltwo)),
\]
and the restrictions of $a^{ij}$ and $b^i$ to $\Gamma_T$ belong to $C_{x,t}^{\alpha,\alpha/2}(\Gamma_T;L_\omega^\infty)$. We denote by $L$ and $L_\alpha$ the corresponding bounds. In particular, $L_\alpha$ bounds the temporal $\alpha/2$-H\"older seminorms of $a^{ij}$ on $\Gamma_T$.
\end{assumption}

\begin{assumption}\label{ass:data}
The free terms
\[
f:\bar Q_T\times\Omega\to\R, \qquad g=(g^k)_{k\ge1}:\bar Q_T\times\Omega\to\elltwo
\]
are predictable and satisfy
\[
f\in C_x^\alpha(\bar Q_T;L_\omega^\gamma), \qquad f|_{\Gamma_T}\in C_{x,t}^{\alpha,\alpha/2}(\Gamma_T;L_\omega^\gamma), 
\qquad f(0,\cdot) \vert_{\partial G} = 0,
\]
and
\[
g\in C_x^{1+\alpha}(\bar Q_T;L_\omega^\gamma(\elltwo)), \qquad g=0 \quad \text{on }\Gamma_T.
\]
\end{assumption}

\begin{remark}\label{rem:corner-compatibility}
The compatibility condition $f(0,\cdot) \vert_{\partial G} = 0$ is the usual corner compatibility condition for zero initial and homogeneous Dirichlet boundary data~\cite{MikuleviciusPragarauskas2003}. It is needed for unweighted $C^{2+\alpha}$-regularity up to $\{0\}\times\partial G$. The condition $g=0$ on $\Gamma_T$ already implies $g(0,\cdot)=0$ on $\partial G$, so no additional corner condition for $g$ is required in the zero-initial case.
\end{remark}

We can now state the main results of the paper. 
The first theorem gives the global Schauder estimate, which is the basic a priori estimate of the paper.

\begin{theorem}[Global Schauder estimate]\label{thm:schauder}
Let $G$ be a domain of class $C^{2+\alpha}$, and let Assumptions~\ref{ass:compatibility}--\ref{ass:data} be satisfied. If $u\in C_{x,t}^{2+\alpha,\alpha/2}(Q_T;L_\omega^\gamma)$ is a quasi-classical solution of \eqref{eq:main}--\eqref{eq:boundary-initial}, then
\begin{equation}\label{eq:schauder-estimate}
|u|_{(2+\alpha,\alpha/2);Q_T}^{L_\omega^\gamma}
\le Ce^{CT}\Bigl(|f|_{\alpha;Q_T}^{L_\omega^\gamma}+|f|_{(\alpha,\alpha/2);\Gamma_T}^{L_\omega^\gamma}+|g|_{1+\alpha;Q_T}^{L_\omega^\gamma(\elltwo)}\Bigr),
\end{equation}
where $C$ depends only on $n$, $\alpha$, $\gamma$, $\kappa$, $K$, $L$, $L_\alpha$, and the $C^{2+\alpha}$-character of $G$.
\end{theorem}

The next theorem shows that the above estimate is sufficient to obtain existence and uniqueness in stochastic H\"older spaces.

\begin{theorem}[Solvability in stochastic H\"older spaces]\label{thm:solvability}
Let $G$ be a domain of class $C^{2+\alpha}$, and let Assumptions~\ref{ass:compatibility}--\ref{ass:data} be satisfied. Then the Dirichlet problem \eqref{eq:main}--\eqref{eq:boundary-initial} admits a unique quasi-classical solution
\[u\in C_{x,t}^{2+\alpha,\alpha/2}(Q_T;L_\omega^\gamma).
\]
Moreover, the solution satisfies the estimate \eqref{eq:schauder-estimate}.
\end{theorem}

The stochastic H\"older theory also yields a pathwise classical solvability result. In the locally bounded classes below, a localizing sequence means increasing stopping times $\tau_m\le T$ such that $\PP(\tau_m=T)\to1$. For $0<\alpha<1$, let $\mathcal C_{l.b.}^{\alpha-}(Q_T)$ denote the space of all predictable random fields $h$ on $\bar Q_T$ for which there is a localizing sequence such that, for every $\beta\in(0,\alpha)$,
\[
\operatorname*{ess\,sup}_{\omega\in\Omega}|h(\omega)|_{\beta;G\times[0,\tau_m(\omega)]}<\infty,
\qquad m\ge 1.
\]
Similarly, $\mathcal C_{l.b.}^{2+\alpha-,\alpha/2-}(Q_T)$ is defined by requiring that for some localizing sequence,
\[
\operatorname*{ess\,sup}_{\omega\in\Omega}|u(\omega)|_{(2+\beta,\beta/2);G\times[0,\tau_m(\omega)]}<\infty
\qquad\text{for every }\beta\in(0,\alpha),\ m\ge1.
\]
Higher spatial orders are defined analogously. For boundary traces we use the analogous notation on $\Gamma_T$, with the parabolic H\"older norms taken on the boundary cylinders.
For Banach-valued random fields, the same notation is used with the corresponding Banach norm.

\begin{theorem}[Classical solvability]\label{thm:classical}
Let $\alpha\in(0,1)$, let $G$ be a bounded domain of class $C^{2+\alpha}$, and let Assumptions~\ref{ass:compatibility}--\ref{ass:parabolicity} hold. Assume that the coefficients and free terms belong locally boundedly to the corresponding pathwise H\"older classes, namely,
\[
\begin{gathered}
a^{ij},\, b^i,\, c,\, f \in \mathcal C_{l.b.}^{\alpha-}(Q_T),
\quad
\sigma^i,\, \nu,\, g \in \mathcal C_{l.b.}^{1+\alpha-}(Q_T;\elltwo),
\\
a^{ij}|_{\Gamma_T},\, b^i|_{\Gamma_T},\,
f|_{\Gamma_T}\in \mathcal C_{l.b.}^{\alpha-,\alpha/2-}(\Gamma_T),
\quad
g|_{\Gamma_T} = 0,
\quad
f|_{\partial G\times\{0\}} = 0 \ \text{a.s.}
\end{gathered}
\]
Then the Dirichlet problem \eqref{eq:main}--\eqref{eq:boundary-initial} admits a unique solution
\[
u\in \mathcal C_{l.b.}^{2+\alpha-,\alpha/2-}(Q_T).
\]
\end{theorem}

\begin{remark}\label{rem:classical-sharp}
Theorem~\ref{thm:classical} may be regarded as a sharp classical solvability result in the sense that it identifies the natural locally bounded pathwise H\"older classes for the free terms and the corresponding solution class. The price paid for passing from stochastic H\"older spaces to pathwise classes is that the argument is nonquantitative: unlike Theorem~\ref{thm:solvability}, Theorem~\ref{thm:classical} does not provide an a priori estimate for the spaces $\mathcal C_{l.b.}^{\alpha-}(Q_T)$ and $\mathcal C_{l.b.}^{2+\alpha-,\alpha/2-}(Q_T)$.
\end{remark}

\begin{remark}\label{rem:initial-data}
The restriction to zero initial data is made for simplicity. Nonzero initial data can be treated by the usual lifting argument, provided the corresponding compatibility conditions at $\{0\}\times\partial G$ are imposed. In the present zero-initial case these reduce to $f(0,\cdot)=0$ and $g(0,\cdot)=0$ on $\partial G$, the latter being included in $g=0$ on $\Gamma_T$.
For nonzero initial data $u_0$ with homogeneous Dirichlet boundary condition, the required compatibility conditions include $u_0|_{\partial G}=0$ and
\[
\bigl(a^{ij}D_{ij}u_0+b^iD_i u_0+cu_0+f(0,\cdot)\bigr)\big|_{\partial G}=0,
\qquad
\bigl(\sigma^{ik}D_i u_0+\nu^k u_0+g^k(0,\cdot)\bigr)\big|_{\partial G}=0.
\]
For $u_0=0$, these reduce to the conditions above.
\end{remark}
\section{The model problem in a half-space}\label{sec:halfspace}

In this section we study a model equation in the half-space
\[
\R^n_+:=\{x=(x_1,x')\in\R^n:x_1>0\},
\qquad
Q_T^+:=\R^n_+\times(0,T],
\qquad
\Sigma_T:=\partial\R^n_+\times[0,T].
\]
The purpose of this section is to establish the boundary H\"older estimate for the model problem, which will serve as the basic local estimate in the proof of Theorem~\ref{thm:schauder}. Throughout this section we write
\begin{equation}\label{eq:model-half-space}
 du=(\bar a^{ij}(t)D_{ij}u+f)\md t+(\bar\sigma^{ik}(t)D_i u+g^k)\md w_t^k
 \qquad \text{in }Q_T^+,
\end{equation}
with the boundary and initial conditions
\begin{equation}\label{eq:model-half-space-bc}
 u=0 \quad \text{on }\Sigma_T,
 \qquad
 u(0,\cdot)=0 \quad \text{in }\R^n_+.
\end{equation}
We assume that $\bar a^{ij}$ and $\bar\sigma^{ik}$ are predictable processes, independent of the spatial variable $x$, satisfy the parabolicity condition \eqref{eq:parabolicity}, and obey the flat-boundary form of the compatibility condition,
\begin{equation}\label{eq:flat-compatibility}
\bar\sigma^{1k}(t)=0,
\qquad k\ge 1.
\end{equation}
Notice that \eqref{eq:flat-compatibility} is exactly the condition \eqref{eq:compatibility} when the boundary is flattened to $\{x_1=0\}$.

The main result of this section is the following estimate.

\begin{proposition}\label{prop:model-main}
Let $\alpha\in(0,1)$ and $\gamma\ge 2$. Assume in addition that
\[
A_\alpha:=\max_{i,j}\sup_{0\le s<t\le T}
\frac{\|\bar a^{ij}(t)-\bar a^{ij}(s)\|_{L_\omega^\infty}}{|t-s|^{\alpha/2}}
<\infty.
\]
Suppose that $u$ is a quasi-classical solution of \eqref{eq:model-half-space}--\eqref{eq:model-half-space-bc} such that
\[
f\in C_x^{\alpha}(\bar Q_T^+;L_\omega^\gamma),
\qquad
f|_{\Sigma_T}\in C_{x,t}^{\alpha,\alpha/2}(\Sigma_T;L_\omega^\gamma),
\qquad
f(0,\cdot)|_{\partial\R_+^n}=0,
\]
and
\[
g\in C_x^{1+\alpha}(\bar Q_T^+;L_\omega^\gamma(\elltwo)),
\qquad
g|_{\Sigma_T} = 0.
\]
Then
\begin{align}\label{eq:model-main-estimate}
|D_x^2u|_{(\alpha,\alpha/2);Q_T^+}^{L_\omega^\gamma}
\le C\Bigl(
|u|_{0;Q_T^+}^{L_\omega^\gamma}
+|f|_{\alpha;Q_T^+}^{L_\omega^\gamma}
+|f|_{(\alpha,\alpha/2);\Sigma_T}^{L_\omega^\gamma}
+|g|_{1+\alpha;Q_T^+}^{L_\omega^\gamma(\elltwo)}
\Bigr),
\end{align}
where $C$ depends only on $n$, $\alpha$, $\gamma$, $\kappa$, $K$, and $A_\alpha$.
\end{proposition}

The proof of Proposition~\ref{prop:model-main} will be completed at the end of the section. We first establish local energy and homogeneous derivative estimates, then estimate the tangential derivatives and the first normal derivative, and finally treat the second normal derivative by means of a boundary lemma for a one-dimensional heat equation on the half-line.

\subsection{Energy estimates and homogeneous local regularity}

For $r>0$ and $x\in\R^n_+$, set
\[
B_r^+(x):=B_r(x)\cap \R^n_+,
\qquad
Q_r^+(t,x):=B_r^+(x)\times(t-r^2,t].
\]
When $(t,x)=(0,0)$, we simply write $B_r^+$ and $Q_r^+$. For a cylinder $Q=G\times I$, Minkowski's inequality gives, since $\gamma\ge2$,
\[
\|v\|_{L_\omega^\gamma L_t^2L_x^2(Q)}
\le\left(\int_Q\|v(t,x)\|_{L_\omega^\gamma}^{2}\,dx\,dt\right)^{1/2}.
\]

\begin{lemma}\label{lem:mixed-norm-half-space}
Let $\gamma\ge2$, $D=B_r(y)\cap\R^n_+$, and $I=(t_0,t_1]$ with $t_1-t_0\le r^2$. For predictable $F^i\in L_\omega^\gamma L_t^2L_x^2(D\times I)$, the problem
\[
 dv=(\bar a^{ij}D_{ij}v+D_iF^i)\md t
       +\bar\sigma^{ik}D_iv\md w_t^k,
 \qquad v|_{\partial D}=0,\qquad v(t_0)=0,
\]
has a unique variational solution. Writing $Q=D\times I$, one has
\begin{equation}\label{eq:half-space-W1p}
\|v\|_{L_\omega^\gamma L_t^\infty L_x^2(Q)}
+\|Dv\|_{L_\omega^\gamma L_t^2L_x^2(Q)}
\le C\sum_{i=1}^n\|F^i\|_{L_\omega^\gamma L_t^2L_x^2(Q)}
\end{equation}
and
\begin{equation}\label{eq:half-space-W2p}
\|v\|_{L_\omega^\gamma L_t^2L_x^2(Q)}
\le Cr\sum_{i=1}^n\|F^i\|_{L_\omega^\gamma L_t^2L_x^2(Q)},
\end{equation}
where $C$ depends only on $n$, $\gamma$, $\kappa$, and $K$.
\end{lemma}

\begin{proof}
The variational construction applies on $D$, with energy space $H_0^1(D)$; compare~\cite[Section~3]{DuLiu2019}. In particular, it does not require the noise vector fields to be tangent to the artificial spherical boundary. Since the coefficients are spatially constant,
\[
\int_D v\bar\sigma^{ik}D_iv\,dx=0.
\]
Thus the martingale in It\^o's formula for $\|v(t)\|_{L_x^2}^2$ vanishes. Integration by parts, parabolicity, and Young's inequality give, almost surely,
\[
\sup_{t\in I}\|v(t)\|_{L_x^2}^2
 +\int_I\|Dv(t)\|_{L_x^2}^2\,dt
 \le C\sum_i\int_I\|F^i(t)\|_{L_x^2}^2\,dt.
\]
Taking $L_\omega^{\gamma/2}$-norms proves \eqref{eq:half-space-W1p}. The Poincar\'e inequality on $D$, obtained by zero extension to $B_r(y)$, proves \eqref{eq:half-space-W2p}. The computation is justified for the variational solution by the usual Galerkin approximation and also gives uniqueness.
\end{proof}

Only homogeneous solutions require pointwise derivative estimates in the approximation argument. The next lemma controls arbitrary tangential derivatives and at most one additional normal derivative.

\begin{lemma}\label{lem:local-pointwise}
Let $\gamma\ge2$, $m\ge0$ be an integer, $\theta\in(0,1)$, and $\beta=1/4$. Suppose that $h$ is a local variational solution of
\[
 dh=\bar a^{ij}D_{ij}h\md t+\bar\sigma^{ik}D_ih\md w_t^k
 \quad\text{in }Q_r^+(s,y),
\]
with zero trace on its flat boundary, and assume \eqref{eq:flat-compatibility}. Then
\[
\begin{aligned}
&|D_{x'}^m h|_{0;Q_{\theta r}^+(s,y)}^{L_\omega^\gamma}
+r|D_xD_{x'}^m h|_{0;Q_{\theta r}^+(s,y)}^{L_\omega^\gamma}
+r^{1+\beta}[D_xD_{x'}^m h]_{\beta;Q_{\theta r}^+(s,y)}^{L_\omega^\gamma}\\
&\qquad\le Cr^{-m-(n+2)/2}\|h\|_{L_\omega^\gamma L_t^2L_x^2(Q_r^+(s,y))},
\end{aligned}
\]
where $C$ depends only on $n$, $\gamma$, $m$, $\theta$, $\kappa$, and $K$. The estimate holds also on cylinders truncated at $t=0$, provided $h(0)=0$ there, with the same constant.
\end{lemma}

\begin{proof}
By parabolic scaling, take $r=1$ and suppress the center in the notation. A localized $L^2$ energy estimate first controls $Dh$ in $L_\omega^\gamma L_t^2L_x^2$. For a further cut-off $v=\zeta h$, the drift forcing belongs to $L_\omega^\gamma L_t^2L_x^2$ and the additive noise $-\bar\sigma^{ik}(D_i\zeta)h$ belongs to $L_\omega^\gamma L_t^2H_0^1$, with norms controlled by $\|h\|_{L_\omega^\gamma L_t^2L_x^2(Q_1^+)}$. For smooth $v$ with compact support and zero flat trace,
\[
\int_{\R_+^n}\Delta v\,\bar a^{ij}D_{ij}v\,dx
 =\sum_k\int_{\R_+^n}\bar a^{ij}D_iD_kv\,D_jD_kv\,dx,
 \qquad
 (v,\bar\sigma^{ik}D_iv)_{H^1}=0.
\]
The boundary terms in the first identity are tangential total derivatives, and the second identity uses $\bar\sigma^{1k}=0$. Testing with $-\Delta v$, using parabolicity, and estimating the additive terms by the Burkholder--Davis--Gundy and Young inequalities yields the $H^1$ estimate for every $\gamma\ge2$. It extends to weak solutions by the $H^1$ variational construction and $L^2$ uniqueness. Iterating with tangential difference quotients, which preserve the zero boundary trace, gives on each fixed smaller cylinder $Q_\rho^+$
\[
\|D_{x'}^\eta h\|_{L_\omega^\gamma L_t^\infty H_x^1(Q_\rho^+)}
+\|D_{x'}^\eta h\|_{L_\omega^\gamma L_t^2H_x^2(Q_\rho^+)}
\le C\|h\|_{L_\omega^\gamma L_t^2L_x^2(Q_1^+)}
\]
for every tangential multi-index $\eta$. At each stage the preceding $H^2$ bound supplies the next tangential difference quotient in the energy class.

The anisotropic Sobolev embedding
$H_{x_1}^1H_{x'}^s\hookrightarrow C_x^\beta$ for $s>(n-1)/2+\beta$, applied to a further cut-off, now bounds the spatial $C^\beta(L_\omega^\gamma)$-norm of every $D_{x'}^\eta h$ by $C\|h\|_{L_\omega^\gamma L_t^2L_x^2(Q_1^+)}$. It remains to obtain one normal derivative. Put $z=D_{x'}^m h$ and $v=\chi z$, where $\chi$ vanishes near the artificial boundary and the initial time and equals one on the desired smaller cylinder. Rewriting the mixed principal terms as a divergence, we obtain
\[
 dv=\bigg(\bar a^{11}D_{11}v+\sum_{i,j\ge2}\bar a^{ij}D_{ij}v
                  +D_iF^i+F^0\bigg)\md t
       +(\bar\sigma^{ik}D_iv+G^k)\md w_t^k,
\]
where
\[
\begin{aligned}
 F^i&=-2\bar a^{ij}(D_j\chi)z
          +\mathbf1_{\{i=1\}}\,2\sum_{j\ge2}\bar a^{1j}D_jv,\\
 F^0&=(\chi_t+\bar a^{ij}D_{ij}\chi)z,
 &G^k&=-\bar\sigma^{ik}(D_i\chi)z.
\end{aligned}
\]
The preceding tangential estimates control all these terms in $C_x^\beta(L_\omega^\gamma)$, and their flat boundary traces vanish. Extend $v,F^0,G$ and $F^i$ for $i\ge2$ oddly, and $F^1$ evenly. The resulting whole-space equation has spatially constant coefficients and satisfies the same parabolicity condition.

The required $C_x^\beta$-estimate for $Dv$ follows from the whole-space Schauder estimate of~\cite[Theorem~1.3]{DuLiu2019}. First convolve the extended equation with a smooth spatial mollifier $\varphi_\varepsilon$ and write $v_\varepsilon=v*\varphi_\varepsilon$. With $R=(1-\Delta)^{-1}$, each smooth $z_{j,\varepsilon}=D_jRv_\varepsilon$ satisfies a whole-space equation whose drift forcing is $D_jR(D_iF^i_\varepsilon+F^0_\varepsilon)\in C_x^\beta$ and whose additive noise is $D_jRG_\varepsilon\in C_x^{1+\beta}$. Spatial Bessel potentials commute with both model operators and are bounded between the indicated Banach-valued H\"older spaces. The Schauder estimate and
\[
 D_jv_\varepsilon=z_{j,\varepsilon}-\Delta z_{j,\varepsilon},
 \qquad |z_{j,\varepsilon}|_0^{L_\omega^\gamma}\le C|v|_0^{L_\omega^\gamma},
\]
give a bound independent of $\varepsilon$. Passing to the limit yields
\[
 |Dv|_\beta^{L_\omega^\gamma}
 \le C\bigg(|v|_0^{L_\omega^\gamma}
       +\sum_i|F^i|_\beta^{L_\omega^\gamma}
       +|F^0|_\beta^{L_\omega^\gamma}+|G|_\beta^{L_\omega^\gamma(\elltwo)}\bigg)
 \le C\|h\|_{L_\omega^\gamma L_t^2L_x^2(Q_1^+)}.
\]
The regularization preserves predictability, and distributional convergence and lower semicontinuity identify the limiting derivatives. For a cylinder truncated at $t=0$, the zero initial condition allows the same cut-offs without a time cut-off at $0$. Rescaling finishes the proof.
\end{proof}

\subsection{H\"older estimate for tangential derivatives}

In this subsection we estimate the tangential derivatives $D_{x'}u=(D_2u,\dots,D_nu)$. 
This part of the argument is close in spirit to the corresponding interior estimate in the whole-space case~\cite[Section~4]{DuLiu2019}, but the localization has to respect the boundary.

\begin{lemma}\label{prop:tangential}
Let $u$ be a quasi-classical solution of \eqref{eq:model-half-space}--\eqref{eq:model-half-space-bc} with $g\equiv0$. Assume that
\[
f=D_i f^i,
\qquad
f^i\in C_x^{\alpha}(\bar Q_T^+;L_\omega^\gamma),
\qquad i=1,\dots,n.
\]
Then
\begin{equation}\label{eq:tangential-estimate}
|D_{x'}u|_{\alpha;Q_T^+}^{L_\omega^\gamma}
\le C\Bigl(
|u|_{0;Q_T^+}^{L_\omega^\gamma}
+\sum_{i=1}^n |f^i|_{\alpha;Q_T^+}^{L_\omega^\gamma}
\Bigr),
\end{equation}
where $C$ depends only on $n$, $\alpha$, $\gamma$, $\kappa$, and $K$.
\end{lemma}

\begin{proof}
We retain the homogeneous approximation scheme of~\cite[Section~4]{DuLiu2019}, using the energy norms from the preceding subsection. Write
\[
 M:=\sum_{i=1}^n|f^i|_{\alpha;Q_T^+}^{L_\omega^\gamma},
 \qquad U_0:=|u|_{0;Q_T^+}^{L_\omega^\gamma}.
\]
For $s\in(0,T]$, $y_1,y_2\in\R_+^n$, and $Y_j=(s,y_j)$, we will show
\begin{equation}\label{eq:tangential-local-goal}
\|D_{x'}u(Y_1)-D_{x'}u(Y_2)\|_{L_\omega^\gamma}
\le C|y_1-y_2|^\alpha(U_0+M).
\end{equation}
We first treat $0<d:=|y_1-y_2|<1/8$. Choose $N_0\ge2$ so that
$2^{-(N_0+2)}\le d<2^{-(N_0+1)}$, and set
\[
 r_N:=2^{-N},\qquad
 Q_N^+(Y):=B_{r_N}^+(y)\times((s-r_N^2)\vee0,s].
\]
The zero initial condition permits the use of Lemma~\ref{lem:local-pointwise} on these truncated cylinders.

For $j=1,2$ and every $N\ge0$, define $u_{j,N}$ by
\begin{equation}\label{eq:tangential-approx}
\begin{cases}
 d u_{j,N}=\bar a^{rs}(t)D_{rs}u_{j,N}\md t
       +\bar\sigma^{rk}(t)D_ru_{j,N}\md w_t^k
       &\text{in }Q_N^+(Y_j),\\[2mm]
 u_{j,N}=u &\text{on }\partial_pQ_N^+(Y_j).
\end{cases}
\end{equation}
Construct $u_{j,N}=u-w_{j,N}$, where $w_{j,N}$ is the zero initial and boundary variational solution from Lemma~\ref{lem:mixed-norm-half-space} with drift forcing $D_i(f^i(t,x)-f^i(t,y_j))$ and no additive noise. By Minkowski's inequality,
\[
\begin{aligned}
&\|f^i(\cdot,\cdot)-f^i(\cdot,y_j)\|_{L_\omega^\gamma L_t^2L_x^2(Q_N^+(Y_j))}\\
&\qquad\le\left(\int_{Q_N^+(Y_j)}
  \|f^i(t,x)-f^i(t,y_j)\|_{L_\omega^\gamma}^{2}\,dx\,dt\right)^{1/2}\\
&\qquad\le C r_N^{\alpha+(n+2)/2}|f^i|_{\alpha;Q_T^+}^{L_\omega^\gamma}.
\end{aligned}
\]
Lemma~\ref{lem:mixed-norm-half-space} therefore yields
\begin{equation}\label{eq:tangential-w-bound}
\|w_{j,N}\|_{L_\omega^\gamma L_t^2L_x^2(Q_N^+(Y_j))}
+r_N\|w_{j,N}\|_{L_\omega^\gamma L_t^\infty L_x^2(Q_N^+(Y_j))}
\le C r_N^{1+\alpha+(n+2)/2}M.
\end{equation}
In particular, the estimate at the terminal time $s$ has size
$\|w_{j,N}(s)\|_{L_\omega^\gamma L_x^2}\le Cr_N^{1+\alpha+n/2}M$.

Let $h_{j,N}=u_{j,N}-u_{j,N+1}$, which is homogeneous in $Q_{N+1}^+(Y_j)$. Lemma~\ref{lem:local-pointwise} with $m=1$, together with \eqref{eq:tangential-w-bound}, gives
\begin{equation}\label{eq:tangential-difference}
|D_{x'}h_{j,N}|_{0;Q_{N+2}^+(Y_j)}^{L_\omega^\gamma}
+r_N|D_xD_{x'}h_{j,N}|_{0;Q_{N+2}^+(Y_j)}^{L_\omega^\gamma}
\le Cr_N^\alpha M.
\end{equation}
Consequently $D_{x'}u_{j,N}(Y_j)$ converges in $L_\omega^\gamma$.

We identify this limit without a pointwise estimate for the nonhomogeneous error $w_{j,N}$. For fixed interior $y_j$ and sufficiently large $N$, obtain $\psi_N$ by scaling a fixed smooth unit-mass spatial averaging kernel to be supported in $B_{r_N/4}(y_j)$. Integration by parts and \eqref{eq:tangential-w-bound} give
\[
\|\langle D_{x'}w_{j,N}(s),\psi_N\rangle\|_{L_\omega^\gamma}
\le Cr_N^{-1-n/2}\|w_{j,N}(s)\|_{L_\omega^\gamma L_x^2}
\le Cr_N^\alpha M.
\]
On the other hand, telescoping \eqref{eq:tangential-difference} from any fixed starting scale, and applying Lemma~\ref{lem:local-pointwise} to that starting approximation, shows that
\[
 r_N|D_xD_{x'}u_{j,N}|_{0;Q_{N+2}^+(Y_j)}^{L_\omega^\gamma}\longrightarrow0.
\]
Thus the value of $D_{x'}u_{j,N}$ at $Y_j$ and its $\psi_N$-average have the same limit. The spatial $L_\omega^\gamma$-continuity of $D_{x'}u$ identifies that limit as $D_{x'}u(Y_j)$. Summing the tail in \eqref{eq:tangential-difference} now proves
\begin{equation}\label{eq:tangential-convergence}
\|D_{x'}u_{j,N}(Y_j)-D_{x'}u(Y_j)\|_{L_\omega^\gamma}
\le Cr_N^\alpha M
\end{equation}
for all $N\ge0$. The unit-scale estimate and summation of the increments also give
\[
 |D_{x'}u|_{0;Q_T^+}^{L_\omega^\gamma}\le C(U_0+M),
\]
since $\|u_{j,0}\|_{L_\omega^\gamma L_t^2L_x^2(Q_0^+(Y_j))}\le C(U_0+M)$. This proves \eqref{eq:tangential-local-goal} for $d\ge1/8$ as well.

For $d<1/8$, decompose
\begin{align}
\|D_{x'}u(Y_1)-D_{x'}u(Y_2)\|_{L_\omega^\gamma}
&\le\|D_{x'}u(Y_1)-D_{x'}u_{1,N_0}(Y_1)\|_{L_\omega^\gamma}
 +\|D_{x'}u(Y_2)-D_{x'}u_{2,N_0}(Y_2)\|_{L_\omega^\gamma}\notag\\
&\quad+\|D_{x'}u_{1,N_0}(Y_1)-D_{x'}u_{1,N_0}(Y_2)\|_{L_\omega^\gamma}\notag\\
&\quad+\|D_{x'}u_{1,N_0}(Y_2)-D_{x'}u_{2,N_0}(Y_2)\|_{L_\omega^\gamma}\notag\\
&=:J_1+J_2+J_3+J_4.
\label{eq:tangential-J}
\end{align}
Equation~\eqref{eq:tangential-convergence} gives $J_1+J_2\le Cd^\alpha M$.
For $J_4$, the inclusion
\[
 Q_{N_0+1}^+(Y_2)
 \subset Q_{N_0}^+(Y_1)\cap Q_{N_0}^+(Y_2)
\]
shows that $u_{1,N_0}-u_{2,N_0}=w_{2,N_0}-w_{1,N_0}$ is homogeneous on the smaller cylinder. Its energy norm is bounded by \eqref{eq:tangential-w-bound}; applying Lemma~\ref{lem:local-pointwise} there with $\theta=1/2$ gives $J_4\le Cr_{N_0}^\alpha M\le Cd^\alpha M$.

To estimate $J_3$, write
\[
 u_{1,N_0}=u_{1,0}+\sum_{N=1}^{N_0}(u_{1,N}-u_{1,N-1}).
\]
By Lemma~\ref{lem:local-pointwise} and the unit-scale energy bound,
\begin{equation}\label{eq:tangential-u10}
|D_xD_{x'}u_{1,0}|_{0;Q_{N_0+1}^+(Y_1)}^{L_\omega^\gamma}
\le C(U_0+M).
\end{equation}
For each increment, apply the same homogeneous lemma directly from $Q_N^+(Y_1)$ to $Q_{N+1}^+(Y_1)$, with the energy bound \eqref{eq:tangential-w-bound}. This retains a strict shrinkage even when $N=N_0$, and gives
\begin{equation}\label{eq:tangential-second-derivative}
|D_xD_{x'}(u_{1,N}-u_{1,N-1})|_{0;Q_{N_0+1}^+(Y_1)}^{L_\omega^\gamma}
\le Cr_N^{\alpha-1}M,
\qquad 1\le N\le N_0.
\end{equation}
The segment from $y_1$ to $y_2$ lies in $B_{r_{N_0}/2}^+(y_1)$. The mean-value theorem and the preceding two estimates imply
\[
 J_3\le Cd(U_0+M)+Cd\sum_{N=1}^{N_0}r_N^{\alpha-1}M
       \le Cd^\alpha(U_0+M).
\]
Combining the four bounds proves \eqref{eq:tangential-local-goal}. Boundary points follow by continuity, and the uniform gradient bound already obtained supplies the zeroth-order part of \eqref{eq:tangential-estimate}.
\end{proof}

\subsection{H\"older estimate for the first normal derivative}

We next estimate the first derivative in the normal direction. In contrast to the tangential derivatives, this requires a suitable reformulation of the equation before one can appeal to a whole-space estimate.

The main idea is to reduce the problem to a whole-space equation by a careful extension argument.
The difficulty is twofold. 
First, since the second-order operator is not diagonal, a direct odd extension is not available.
Second, in the present stochastic H\"older framework, it is not convenient to use a stochastic-flow transformation to remove the stochastic first-order part (in the spirit of Krylov's analytic approach; see~\cite{Krylov1999Analytic}).
Indeed, our H\"older norms are $L_\omega^\gamma$-valued, with the expectation taken before the H\"older quotient, and a random change of variables would modify these norms in a way that is not transparent.
We therefore proceed differently: we first rewrite the equation in a form adapted to the flat boundary, then extend it across $\{x_1=0\}$ by a suitable odd/even extension, and finally apply the whole-space estimate.

\begin{lemma}\label{prop:first-normal}
Let $u$ be a quasi-classical solution of \eqref{eq:model-half-space}--\eqref{eq:model-half-space-bc} with $g\equiv0$. Assume that
\[
f=D_i f^i,
\qquad
f^i\in C_x^{\alpha}(\bar Q_T^+;L_\omega^\gamma).
\]
Then
\begin{equation}\label{eq:first-normal-estimate}
|D_xu|_{(\alpha,\alpha/2);Q_T^+}^{L_\omega^\gamma}
\le C\Bigl(
|u|_{0;Q_T^+}^{L_\omega^\gamma}
+\sum_{i=1}^n |f^i|_{\alpha;Q_T^+}^{L_\omega^\gamma}
\Bigr),
\end{equation}
where $C$ depends only on $n$, $\alpha$, $\gamma$, $\kappa$, and $K$.
\end{lemma}

\begin{proof}
We divide the proof into three steps.

\medskip
\textit{Step 1. A reformulation adapted to the flat boundary.}
For $i=2,\dots,n$, define
\[
F^i(t,x):=f^i(t,0,x_2,\dots,x_{i-1},x_i+x_1,x_{i+1},\dots,x_n),
\qquad (t,x)\in \bar Q_T^+,
\]
and set
\begin{equation}\label{eq:first-normal-forcing}
\widetilde f^1:=f^1+2\sum_{i=2}^n \bar a^{1i}D_i u+\sum_{i=2}^n F^i,
\qquad
\widetilde f^i:=f^i-F^i,
\quad i=2,\dots,n.
\end{equation}
Since $F^i$ depends on $x_1$ and $x_i$ only through the combination $x_i+x_1$, one has, in the distributional sense,
\begin{equation}\label{eq:F-identity}
D_1F^i=D_iF^i,
\qquad i=2,\dots,n.
\end{equation}
Using \eqref{eq:F-identity} and the symmetry $\bar a^{1i}=\bar a^{i1}$, we compute
\begin{align*}
D_i\widetilde f^i
&=D_1f^1+2\sum_{i=2}^n \bar a^{1i}D_{1i}u+\sum_{i=2}^n D_1F^i
+\sum_{i=2}^n (D_if^i-D_iF^i) \\
&=D_i f^i+2\sum_{i=2}^n \bar a^{1i}D_{1i}u.
\end{align*}
Therefore $u$ satisfies
\begin{equation}\label{eq:first-normal-rewritten}
du=\Bigl(\bar a^{11}D_{11}u+\sum_{i,j=2}^n \bar a^{ij}D_{ij}u+D_i\widetilde f^i\Bigr)\md t
+\sum_{i=2}^n \bar\sigma^{ik}D_i u\md w_t^k
\quad\text{in }Q_T^+.
\end{equation}
Moreover, for each $i\ge2$,
\[
\widetilde f^i(t,0,x')=f^i(t,0,x')-F^i(t,0,x')=0,
\]
which is the crucial boundary cancellation needed below.

\medskip
\textit{Step 2. Extension to the whole space.}
We extend $u$ oddly in the normal variable and write
\[
v(t,x_1,x'):=
\begin{cases}
 u(t,x_1,x'), & x_1\ge0,\\
-u(t,-x_1,x'), & x_1<0.
\end{cases}
\]
Next, for $i=2,\dots,n$ we extend $\widetilde f^i$ oddly, while $\widetilde f^1$ is extended evenly:
\[
\begin{aligned}
\widehat f^1(t,x_1,x') & :=
\begin{cases}
 \widetilde f^1(t,x_1,x'), & x_1\ge0,\\
 \widetilde f^1(t,-x_1,x'), & x_1<0,
\end{cases}\\
\widehat f^i(t,x_1,x') & :=
\begin{cases}
 \widetilde f^i(t,x_1,x'), & x_1\ge0,\\
-\widetilde f^i(t,-x_1,x'), & x_1<0,
\end{cases}
\qquad i=2,\dots,n.
\end{aligned}
\]
Because $\widetilde f^i(t,0,x')=0$ for $i\ge2$, the odd extensions $\widehat f^i$ remain in
$C_x^\alpha(\R^n\times(0,T];L_\omega^\gamma)$. The even extension $\widehat f^1$ is also of class
$C_x^\alpha(\R^n\times(0,T];L_\omega^\gamma)$. Hence $v$ satisfies the whole-space equation
\begin{equation}\label{eq:first-normal-whole-space}
dv=\Bigl(\bar a^{11}D_{11}v+\sum_{i,j=2}^n \bar a^{ij}D_{ij}v+D_i\widehat f^i\Bigr)\md t
+\sum_{i=2}^n \bar\sigma^{ik}D_i v\md w_t^k
\quad\text{in }\R^n\times(0,T],
\end{equation}
with zero initial condition.

Apply the Bessel-potential argument from Lemma~\ref{lem:local-pointwise}, with $\beta$ replaced by $\alpha$, using the parabolic estimate of~\cite[Corollary~4.2]{DuLiu2019}. In the notation of that argument, $|z_{j,\varepsilon}|_0^{L_\omega^\gamma}\le C|v|_0^{L_\omega^\gamma}$ and $D^2z_{j,\varepsilon}$ is controlled in $C_{x,t}^{\alpha,\alpha/2}$, uniformly in $\varepsilon$ and $T$. Spatial interpolation, the equation, and the Burkholder--Davis--Gundy inequality then give a temporal $1/2$-H\"older bound for $z_{j,\varepsilon}$ on time increments at most one; larger increments are controlled by its supremum norm. Since $D_jv_\varepsilon=z_{j,\varepsilon}-\Delta z_{j,\varepsilon}$, passing to the limit in the spatial regularization and restricting to the half-space yields
\begin{equation}\label{eq:first-normal-whole-space-estimate}
|D_xu|_{(\alpha,\alpha/2);Q_T^+}^{L_\omega^\gamma}
\le C\Bigl(
|u|_{0;Q_T^+}^{L_\omega^\gamma}
+\sum_{i=1}^n |\widetilde f^i|_{\alpha;Q_T^+}^{L_\omega^\gamma}
\Bigr).
\end{equation}

\medskip
\textit{Step 3. Estimate of the modified forcing terms.}
For $i=2,\dots,n$, the shear defining $F^i$ is bounded in $C_x^\alpha$, with a bound depending only on $\kappa$ and $K$, and therefore
\[
|\widetilde f^i|_{\alpha;Q_T^+}^{L_\omega^\gamma}
\le |f^i|_{\alpha;Q_T^+}^{L_\omega^\gamma}+|F^i|_{\alpha;Q_T^+}^{L_\omega^\gamma}
\le C|f^i|_{\alpha;Q_T^+}^{L_\omega^\gamma}.
\]
For $\widetilde f^1$, using \eqref{eq:first-normal-forcing}, we obtain
\begin{align*}
|\widetilde f^1|_{\alpha;Q_T^+}^{L_\omega^\gamma}
&\le |f^1|_{\alpha;Q_T^+}^{L_\omega^\gamma}
+ C\sum_{i=2}^n |D_i u|_{\alpha;Q_T^+}^{L_\omega^\gamma}
+ C\sum_{i=2}^n |F^i|_{\alpha;Q_T^+}^{L_\omega^\gamma} \\
&\le C\Bigl(
\sum_{i=1}^n |f^i|_{\alpha;Q_T^+}^{L_\omega^\gamma}
+ |D_{x'}u|_{\alpha;Q_T^+}^{L_\omega^\gamma}
\Bigr).
\end{align*}
Finally, Lemma~\ref{prop:tangential} yields
\[
|D_{x'}u|_{\alpha;Q_T^+}^{L_\omega^\gamma}
\le C\Bigl(
|u|_{0;Q_T^+}^{L_\omega^\gamma}
+\sum_{i=1}^n |f^i|_{\alpha;Q_T^+}^{L_\omega^\gamma}
\Bigr).
\]
Substituting the above bounds into \eqref{eq:first-normal-whole-space-estimate}, we obtain
\eqref{eq:first-normal-estimate}. The proof is complete.
\end{proof}

\subsection{A boundary lemma on the half-line}\label{subsec:half-line-lemma}

The main remaining difficulty is the estimate of the second normal derivative. After the reductions in the previous subsection, this difficulty is reduced to a one-dimensional heat equation on the half-line with time-dependent boundary data. The estimate below is a key new ingredient of the paper.

\begin{lemma}\label{lem:half-line-boundary}
Let $\alpha\in(0,1)$, let $h\in C^{1+\alpha/2}([0,T];L_\omega^\gamma)$ satisfy
\[
h(0)=0,
\qquad
h'(0)=0,
\]
and let $v$ be the solution of
\begin{equation}\label{eq:half-line-problem}
\partial_t v(t,y)=\partial_y^2 v(t,y),
\qquad
v(0,y)=0,
\qquad
v(t,0)=h(t).
\end{equation}
Then
\begin{equation}\label{eq:half-line-estimate}
|\partial_t v|_{(\alpha,\alpha/2);[0,T]\times\R_+}^{L_\omega^\gamma}
+|\partial_y^2 v|_{(\alpha,\alpha/2);[0,T]\times\R_+}^{L_\omega^\gamma}
\le C\,[h']_{\alpha/2;[0,T]}^{L_\omega^\gamma},
\end{equation}
where $C=C(\alpha,\gamma,T)$. 

Moreover, if $v_1$ and $v_2$ correspond to boundary data $h_1$ and $h_2$, respectively, and if $h_1'(0)=h_2'(0)=0$, then
\begin{equation}\label{eq:half-line-stability}
\sup_{(t,y)\in[0,T]\times\R_+}
\E\bigl|\partial_y^2 v_1(t,y)-\partial_y^2 v_2(t,y)\bigr|^\gamma
\le C\sup_{t\in[0,T]}\E\bigl|h_1'(t)-h_2'(t)\bigr|^\gamma.
\end{equation}
\end{lemma}

\begin{proof}
Let
\[
P(s,y):=\frac{y}{2\sqrt{\pi}\,s^{3/2}}\exp\!\Bigl(-\frac{y^2}{4s}\Bigr),
\qquad s>0,
\ y>0,
\]
be the Poisson kernel for the heat equation on the half-line. Then the solution admits the representation
\begin{equation}\label{eq:half-line-poisson}
v(t,y)=\int_0^t P(s,y)h(t-s)\md s,
\qquad (t,y)\in [0,T]\times \R_+.
\end{equation}
For $y>0$, the change of variables $s=y^2r$ yields
\begin{equation}\label{eq:half-line-scaled-v13}
v(t,y)=\int_0^{t/y^2} P(r,1)h(t-y^2r)\md r,
\qquad
P(r,1):=\frac{1}{2\sqrt{\pi}\,r^{3/2}}e^{-1/(4r)}.
\end{equation}
Since $h(0)=0$, differentiation with respect to $t$ gives
\begin{equation}\label{eq:half-line-dtv-v13}
\partial_t v(t,y)=\int_0^{t/y^2} P(r,1)h'(t-y^2r)\md r,
\qquad t>0,
\ y>0.
\end{equation}
On the other hand, $\partial_t v(t,0)=h'(t)$ for $t\in[0,T]$, and of course
\[
\partial_y^2 v=\partial_t v
\qquad\text{in }(0,T]\times \R_+.
\]
Therefore it suffices to estimate $\partial_t v$.

We now extend $h'$ to $(-\infty,T]$ by setting
\[
\widetilde h'(t):=0,
\qquad t<0.
\]
Because $h'(0)=0$, this extension belongs to $C^{\alpha/2}(( -\infty,T];L_\omega^\gamma)$ and
\begin{equation}\label{eq:half-line-extension-v13}
[\widetilde h']_{\alpha/2;(-\infty,T]}^{L_\omega^\gamma}
\le C [h']_{\alpha/2;[0,T]}^{L_\omega^\gamma}.
\end{equation}
Using \eqref{eq:half-line-dtv-v13} and the extension, we may write for all $(t,y)\in[0,T]\times \R_+$,
\begin{equation}\label{eq:half-line-dtv-extended-v13}
\partial_t v(t,y)=\int_0^\infty P(r,1)\widetilde h'(t-y^2r)\md r.
\end{equation}
Set
\[
\gamma':=\frac{\gamma}{\gamma-1}\in(1,2],
\qquad
\varepsilon:=\frac{1-\alpha}{4}\in\Bigl(0,\frac14\Bigr).
\]
\noindent Throughout this proof, $C$ denotes a positive constant depending only on $\alpha$ and $\gamma$, and may vary from line to line.

We first prove the H\"older continuity in time. Let $t_1,t_2\in[0,T]$ and $y\ge 0$. The case $y=0$ is immediate from $\partial_t v(t,0)=h'(t)$. For $y>0$, \eqref{eq:half-line-dtv-extended-v13} gives
\begin{align*}
\partial_t v(t_1,y)-\partial_t v(t_2,y)
&=\int_0^\infty P(r,1)
   \bigl[\widetilde h'(t_1-y^2r)-\widetilde h'(t_2-y^2r)\bigr]dr \\
&=: I_1+I_2,
\end{align*}
where
\begin{align*}
I_1&:=\int_0^1 P(r,1)
      \bigl[\widetilde h'(t_1-y^2r)-\widetilde h'(t_2-y^2r)\bigr] \md r,\\
I_2&:=\int_1^\infty P(r,1)
      \bigl[\widetilde h'(t_1-y^2r)-\widetilde h'(t_2-y^2r)\bigr]\md r.
\end{align*}
In $I_2$ we use the change of variables $r=s^{-1}$ and obtain
\[
I_2=\int_0^1 s^{-2}P(s^{-1},1)
    \bigl[\widetilde h'(t_1-y^2/s)-\widetilde h'(t_2-y^2/s)\bigr]\md s.
\]
By H\"older's inequality,
\begin{align*}
\mathbb{E}|I_1|^\gamma
&\le \biggl( \int_0^1 |P(r,1)|^{\gamma'}\md r \biggr)^{\gamma/\gamma'}
      \int_0^1 \mathbb{E}\bigl|\widetilde h'(t_1-y^2r)-\widetilde h'(t_2-y^2r)\bigr|^\gamma \md r,\\
\mathbb{E}|I_2|^\gamma
&\le \biggl( \int_0^1 |s^{\varepsilon-3/2-1/\gamma'}P(s^{-1},1)|^{\gamma'} \md s \biggr)^{\gamma/\gamma'} \\
&\qquad \times \int_0^1 s^{\gamma/2-1-\gamma\varepsilon}
      \mathbb{E}\bigl|\widetilde h'(t_1-y^2/s)-\widetilde h'(t_2-y^2/s)\bigr|^\gamma \md s.
\end{align*}
Since \eqref{eq:half-line-extension-v13} yields
\[
\mathbb{E}\bigl|\widetilde h'(t_1-\theta)-\widetilde h'(t_2-\theta)\bigr|^\gamma
\le C[h']_{\alpha/2;[0,T]}^\gamma |t_1-t_2|^{\gamma\alpha/2}
\qquad \text{for all }\theta\ge 0,
\]
and
\[
\int_0^1 |P(r,1)|^{\gamma'}\md r<\infty,
\qquad
\int_0^1 |s^{\varepsilon-3/2-1/\gamma'} P(s^{-1},1)|^{\gamma'}\md s<\infty,
\qquad
\int_0^1 s^{\gamma/2-1-\gamma\varepsilon}\md s<\infty,
\]
we conclude that
\begin{equation}\label{eq:half-line-time-holder-v13}
\mathbb{E}\bigl|\partial_t v(t_1,y)-\partial_t v(t_2,y)\bigr|^\gamma
\le C[h']_{\alpha/2;[0,T]}^\gamma |t_1-t_2|^{\gamma\alpha/2}.
\end{equation}

We next estimate the H\"older continuity in the spatial variable. Let $t\in[0,T]$ and $y_1,y_2\ge 0$. Set $d:=|y_1-y_2|$ and assume, without loss of generality, that $y_1\ge y_2$.

\smallskip
\emph{Case 1: $y_1+y_2\le 2d$.}
Using \eqref{eq:half-line-dtv-extended-v13}, we write
\[
\partial_t v(t,y_1)-\partial_t v(t,y_2)
=\int_0^\infty P(r,1)
  \bigl[\widetilde h'(t-y_1^2r)-\widetilde h'(t-y_2^2r)\bigr]\md r.
\]
As above, splitting the integral into $(0,1)$ and $(1,\infty)$, performing the change of variables $r=s^{-1}$ in the second part, and then applying H\"older's inequality, we obtain
\begin{align*}
\mathbb{E}\bigl|\partial_t v(t,y_1)-\partial_t v(t,y_2)\bigr|^\gamma
&\le C[h']_{\alpha/2;[0,T]}^\gamma |y_1^2-y_2^2|^{\gamma\alpha/2} \\
&\qquad \times \biggl(
\int_0^1 r^{\gamma\alpha/2}\md r+
\int_0^1 s^{\gamma(1-\alpha)/2-1-\gamma\varepsilon}\md s
\biggr).
\end{align*}
Both integrals are finite, and since $y_1+y_2\le 2d$,
\[
|y_1^2-y_2^2|=(y_1+y_2)d\le 2d^2.
\]
Hence
\begin{equation}\label{eq:half-line-space-case1-v13}
\mathbb{E}\bigl|\partial_t v(t,y_1)-\partial_t v(t,y_2)\bigr|^\gamma
\le C[h']_{\alpha/2;[0,T]}^\gamma d^{\gamma\alpha}.
\end{equation}

\smallskip
\emph{Case 2: $y_1+y_2>2d$.}
Using \eqref{eq:half-line-poisson} and the identity
\[
\int_0^\infty P(s,y)\md s=1,
\qquad y>0,
\]
we write
\begin{equation}\label{eq:half-line-cancel-v13}
\partial_t v(t,y_1)-\partial_t v(t,y_2)
=\int_0^\infty \bigl(P(s,y_1)-P(s,y_2)\bigr)
  \bigl(\widetilde h'(t-s)-\widetilde h'(t)\bigr)\md s.
\end{equation}
By the mean value theorem,
\[
|P(s,y_1)-P(s,y_2)|\le d\sup_{y\in[y_2,y_1]}|\partial_yP(s,y)|.
\]
A direct differentiation gives
\[
\partial_yP(s,y)=\frac{1}{2\sqrt{\pi}\,s^{3/2}}e^{-y^2/(4s)}
\biggl(1-\frac{y^2}{2s}\biggr),
\]
so that
\begin{equation}\label{eq:half-line-kernel-derivative-v13}
|\partial_yP(s,y)|
\le C s^{-3/2}e^{-y^2/(8s)}\biggl(1+\frac{y^2}{s}\biggr).
\end{equation}
In this case $y_2>d/2$ and $y_1<3y_2$. Hence
\[
\sup_{y\in[y_2,y_1]}|\partial_yP(s,y)|
\le C s^{-3/2}e^{-y_2^2/(8s)}\biggl(1+\frac{y_2^2}{s}\biggr).
\]
Applying Minkowski's integral inequality to \eqref{eq:half-line-cancel-v13}, we obtain
\begin{align*}
\bigl(\mathbb{E}|\partial_t v(t,y_1)-\partial_t v(t,y_2)|^\gamma\bigr)^{1/\gamma}
&\le d\int_0^\infty\sup_{y\in[y_2,y_1]}|\partial_yP(s,y)|
      \bigl(\mathbb{E}|\widetilde h'(t-s)-\widetilde h'(t)|^\gamma\bigr)^{1/\gamma}\md s \\
&\le C d[h']_{\alpha/2;[0,T]}
      \int_0^\infty s^{(\alpha-3)/2}e^{-y_2^2/(8s)}
      \biggl(1+\frac{y_2^2}{s}\biggr)\md s.
\end{align*}
The change of variables $s=y_2^2r$ gives
\begin{align*}
\int_0^\infty s^{(\alpha-3)/2}e^{-y_2^2/(8s)}
      \biggl(1+\frac{y_2^2}{s}\biggr)\md s
&=y_2^{\alpha-1}
      \int_0^\infty r^{(\alpha-3)/2}e^{-1/(8r)}\biggl(1+\frac1r\biggr)\md r \\
&\le C y_2^{\alpha-1},
\end{align*}
where the last integral is finite because $\alpha\in(0,1)$. Since $y_2>d/2$ and $\alpha<1$, we have $dy_2^{\alpha-1}\le Cd^\alpha$. Therefore
\begin{equation}\label{eq:half-line-space-case2-v13}
\mathbb{E}\bigl|\partial_t v(t,y_1)-\partial_t v(t,y_2)|^\gamma
\le C[h']_{\alpha/2;[0,T]}^\gamma d^{\gamma\alpha}.
\end{equation}

Combining \eqref{eq:half-line-space-case1-v13} and \eqref{eq:half-line-space-case2-v13}, we conclude that
\begin{equation}\label{eq:half-line-space-holder-v13}
\mathbb{E}\bigl|\partial_t v(t,y_1)-\partial_t v(t,y_2)\bigr|^\gamma
\le C[h']_{\alpha/2;[0,T]}^\gamma |y_1-y_2|^{\gamma\alpha}.
\end{equation}
The identity \eqref{eq:half-line-dtv-extended-v13} also gives $|\partial_t v|_0^{L_\omega^\gamma}\le |h'|_0^{L_\omega^\gamma}\le T^{\alpha/2}[h']_{\alpha/2;[0,T]}^{L_\omega^\gamma}$. Together with \eqref{eq:half-line-time-holder-v13}, we obtain
\[
|\partial_t v|_{(\alpha,\alpha/2);[0,T]\times\R_+}^{L_\omega^\gamma}
\le C(1+T^{\alpha/2})[h']_{\alpha/2;[0,T]}^{L_\omega^\gamma}.
\]
Since $\partial_y^2v=\partial_t v$, the estimate \eqref{eq:half-line-estimate} follows.
Finally, let $v_1$ and $v_2$ correspond to boundary data $h_1$ and $h_2$, respectively. Applying \eqref{eq:half-line-dtv-extended-v13} to the difference $v_1-v_2$, we obtain
\begin{align*}
\Bigl(\mathbb{E}\bigl|\partial_y^2v_1(t,y)-\partial_y^2v_2(t,y)\bigr|^\gamma\Bigr)^{1/\gamma}
&=\Bigl(\mathbb{E}\bigl|\partial_t (v_1-v_2)(t,y)\bigr|^\gamma\Bigr)^{1/\gamma} \\
&\le \int_0^\infty P(r,1)
   \Bigl(\mathbb{E}\bigl|(\widetilde{h}_1'-\widetilde{h}_2')(t-y^2r)\bigr|^\gamma\Bigr)^{1/\gamma}dr \\
&\le \sup_{s\in[0,T]}\Bigl(\mathbb{E}\bigl|h_1'(s)-h_2'(s)\bigr|^\gamma\Bigr)^{1/\gamma},
\end{align*}
where $\widetilde{h}_1'-\widetilde{h}_2'$ denotes the extension of $h_1'-h_2'$ by $0$ to $(-\infty,0)$. Taking the supremum over $(t,y)\in[0,T]\times\R_+$ proves \eqref{eq:half-line-stability}.
\end{proof}

\subsection{The second normal derivative and completion of the proof}

We now complete the proof of Proposition~\ref{prop:model-main}. The estimates obtained in the previous subsections already control the mixed derivatives and the first normal derivative. It remains to estimate $D_{11}u$. The point of the argument below is that, after removing the stochastic term and then the boundary trace of the reduced forcing, the remaining equation becomes a one-dimensional Dirichlet problem in the normal variable, to which an odd-extension argument can be applied.

We begin with the case $g\equiv0$.

\begin{lemma}\label{lem:model-g-zero}
Let $u$ be a quasi-classical solution of \eqref{eq:model-half-space}--\eqref{eq:model-half-space-bc} with $g\equiv0$. Suppose that the coefficients $\bar a^{ij}$ satisfy the temporal H\"older condition of Proposition~\ref{prop:model-main}, with bound $A_\alpha$. Assume that
\[
f\in C_x^{\alpha}(\bar Q_T^+;L_\omega^\gamma),
\qquad
f|_{\Sigma_T}\in C_{x,t}^{\alpha,\alpha/2}(\Sigma_T;L_\omega^\gamma),
\qquad
f(0,\cdot)|_{\partial\R_+^n}=0.
\]
Then
\begin{equation}\label{eq:model-g-zero-estimate}
|D_x^2u|_{(\alpha,\alpha/2);Q_T^+}^{L_\omega^\gamma}
\le C\Bigl(
|u|_{0;Q_T^+}^{L_\omega^\gamma}
+|f|_{\alpha;Q_T^+}^{L_\omega^\gamma}
+|f|_{(\alpha,\alpha/2);\Sigma_T}^{L_\omega^\gamma}
\Bigr),
\end{equation}
where $C$ depends only on $n$, $\alpha$, $\gamma$, $\kappa$, $K$, and $A_\alpha$.
\end{lemma}

\begin{proof}
Throughout this proof, $C$ may depend on $n$, $\alpha$, $\gamma$, $\kappa$, $K$, and $A_\alpha$. Step~1 holds for arbitrary $T$. In Step~2 we first assume $T\le1$ and then use time cut-offs.

\medskip
\textit{Step 1. Estimate of the mixed derivatives.}
For $\ell=2,\dots,n$ and $h\neq0$, define
\[
\delta_h^\ell u(t,x):=\frac{u(t,x+he_\ell)-u(t,x)}{h}.
\]
Since the translation is tangential to the flat boundary, $\delta_h^\ell u$ still vanishes on $\Sigma_T$. Moreover,
\[
d\,\delta_h^\ell u
=\bigl(\bar a^{ij}(t)D_{ij}\delta_h^\ell u+\delta_h^\ell f\bigr)dt
+\bar\sigma^{ik}(t)D_i\delta_h^\ell u\md w_t^k.
\]
Writing
\[
\delta_h^\ell f=D_\ell F_h,
\qquad
F_h(t,x):=\int_0^1 f(t,x+\theta he_\ell)\md \theta,
\]
we may apply Lemma~\ref{prop:first-normal} to $\delta_h^\ell u$. Since
\[
|F_h|_{\alpha;Q_T^+}^{L_\omega^\gamma}
\le |f|_{\alpha;Q_T^+}^{L_\omega^\gamma},
\qquad
|\delta_h^\ell u|_{0;Q_T^+}^{L_\omega^\gamma}
\le |D_\ell u|_{0;Q_T^+}^{L_\omega^\gamma},
\]
we obtain
\begin{equation}\label{eq:mixed-difference-quotient}
|D_x\delta_h^\ell u|_{(\alpha,\alpha/2);Q_T^+}^{L_\omega^\gamma}
\le C\Bigl(
|D_\ell u|_{0;Q_T^+}^{L_\omega^\gamma}
+|f|_{\alpha;Q_T^+}^{L_\omega^\gamma}
\Bigr).
\end{equation}
For a uniform gradient bound, repeat the single-center construction of Lemma~\ref{prop:tangential} with drift forcing $f$. For the errors $w_N=u-u_N$, the energy estimate and the Poincar\'e inequality give
\[
 \|w_N\|_{L_\omega^\gamma L_t^2L_x^2(Q_N^+(s,y))}
 +r_N\|w_N\|_{L_\omega^\gamma L_t^\infty L_x^2(Q_N^+(s,y))}
 \le Cr_N^{2+(n+2)/2}|f|_{0;Q_T^+}^{L_\omega^\gamma}.
\]
Lemma~\ref{lem:local-pointwise} with $m=0$ consequently yields, on $Q_{N+2}^+(s,y)$,
\[
 |D(u_N-u_{N+1})|_0^{L_\omega^\gamma}
 +r_N^\beta[D(u_N-u_{N+1})]_\beta^{L_\omega^\gamma}
 \le Cr_N|f|_{0;Q_T^+}^{L_\omega^\gamma},
 \qquad \beta=1/4.
\]
The gradient increments and their $C^\beta$-seminorms are summable. As in Lemma~\ref{prop:tangential}, spatial averaging identifies their limit as $Du$: the averaged error $Dw_N$ is at most $Cr_N|f|_0$, and the uniform local $C^\beta$-bound makes the point value and average converge to the same limit. Summing with the unit-scale estimate and extending to the flat boundary gives
\begin{equation}\label{eq:first-derivative-model-g-zero}
|Du|_{0;Q_T^+}^{L_\omega^\gamma}
\le C\Bigl(
|u|_{0;Q_T^+}^{L_\omega^\gamma}
+|f|_{\alpha;Q_T^+}^{L_\omega^\gamma}
\Bigr).
\end{equation}
Letting $h\to0$ in \eqref{eq:mixed-difference-quotient}, we infer that
\begin{equation}\label{eq:mixed-derivative-estimate}
|D_xD_{x'}u|_{(\alpha,\alpha/2);Q_T^+}^{L_\omega^\gamma}
\le C\Bigl(
|u|_{0;Q_T^+}^{L_\omega^\gamma}
+|f|_{\alpha;Q_T^+}^{L_\omega^\gamma}
\Bigr).
\end{equation}

\medskip
\textit{Step 2. Estimate of $D_{11}u$.}
The remaining difficulty is the second derivative in the normal direction. We proceed by two reductions. We first remove the stochastic term, thereby reducing the problem to a one-dimensional equation in the normal variable. The reduced forcing still has a nonzero boundary trace, which prevents a direct odd extension. We therefore split off this boundary trace by a further correction term.

Let $U$ be the solution of
\begin{equation}\label{eq:U-equation}
\begin{cases}
dU=\Delta U\md t+\bar\sigma^{ik}(t)D_i u\md w_t^k &\text{in }Q_T^+,\\
U=0 &\text{on }\Sigma_T,\\
U(0,\cdot)=0 &\text{in }\R^n_+.
\end{cases}
\end{equation}
Since $\bar\sigma^{1k}=0$ and $u=0$ on $\Sigma_T$, the stochastic forcing in \eqref{eq:U-equation} vanishes on $\Sigma_T$. By odd extension across $\{x_1=0\}$ and the whole-space Schauder estimate for the stochastic heat equation, see~\cite[Theorem~1.3]{DuLiu2019},
\begin{equation}\label{eq:U-estimate}
\begin{aligned}
|U|_{(2+\alpha,\alpha/2);Q_T^+}^{L_\omega^\gamma}
&\le C\,|\bar\sigma^{ik}D_i u|_{1+\alpha;Q_T^+}^{L_\omega^\gamma(\elltwo)}
\le C\Bigl(
|D_xD_{x'}u|_{(\alpha,\alpha/2);Q_T^+}^{L_\omega^\gamma}
+|Du|_{0;Q_T^+}^{L_\omega^\gamma}
\Bigr)\\
&\le C\Bigl(
|u|_{0;Q_T^+}^{L_\omega^\gamma}
+|f|_{\alpha;Q_T^+}^{L_\omega^\gamma}
\Bigr),
\end{aligned}
\end{equation}
where in the last step we used \eqref{eq:first-derivative-model-g-zero} and \eqref{eq:mixed-derivative-estimate}.

Set $\widetilde u:=u-U$. Then $\widetilde u$ satisfies the one-dimensional equation
\begin{equation}\label{eq:utilde-equation}
\partial_t\widetilde u
=\bar a^{11}(t)D_{11}\widetilde u+\widetilde f
\qquad\text{in }Q_T^+,
\end{equation}
with zero boundary and initial conditions, where
\begin{equation}\label{eq:tildef-equation}
\widetilde f
:=f
+\sum_{(i,j)\neq(1,1)}\bar a^{ij}(t)D_{ij}u
+\bigl(\bar a^{11}(t)-1\bigr)D_{11}U
-\sum_{i=2}^n D_{ii}U.
\end{equation}
Using \eqref{eq:U-estimate} and \eqref{eq:mixed-derivative-estimate}, we obtain
\begin{align}
|\widetilde f|_{\alpha;Q_T^+}^{L_\omega^\gamma}
&\le C\Bigl(
|u|_{0;Q_T^+}^{L_\omega^\gamma}
+|f|_{\alpha;Q_T^+}^{L_\omega^\gamma}
\Bigr),
\label{eq:tildef-alpha-bound}
\\
|\widetilde f|_{(\alpha,\alpha/2);\Sigma_T}^{L_\omega^\gamma}
&\le C\Bigl(
|u|_{0;Q_T^+}^{L_\omega^\gamma}
+|f|_{\alpha;Q_T^+}^{L_\omega^\gamma}
+|f|_{(\alpha,\alpha/2);\Sigma_T}^{L_\omega^\gamma}
\Bigr).
\label{eq:tildef-boundary-bound}
\end{align}
The boundary estimate uses the H\"older product rule and the temporal bound $A_\alpha$ for $\bar a^{ij}$.

Since $\widetilde f$ does not vanish on $\Sigma_T$, the equation for $\widetilde u$ is not yet suitable for an odd extension. We therefore remove the boundary trace of $\widetilde f$. Write
\[
b(t,x'):=\bigl(\bar a^{11}(t)\bigr)^{-1}\widetilde f(t,0,x').
\]
By uniform parabolicity and the H\"older product rule,
\[
\bigl[(\bar a^{11})^{-1}\bigr]_{\alpha/2;[0,T]}^{L_\omega^\infty}
\le C(\kappa)A_\alpha,
\qquad
|b|_{(\alpha,\alpha/2);\Sigma_T}^{L_\omega^\gamma}
\le C(\kappa,K,A_\alpha)
|\widetilde f|_{(\alpha,\alpha/2);\Sigma_T}^{L_\omega^\gamma}.
\]
Let $V$ be the solution of
\begin{equation}\label{eq:V-equation}
\begin{cases}
\partial_tV=\partial_{11}V+b(t,x') &\text{in }Q_T^+,\\
V=0 &\text{on }\Sigma_T,\\
V(0,\cdot)=0 &\text{in }\R^n_+.
\end{cases}
\end{equation}
Set
\[
H(t,x'):=\int_0^t b(s,x')\md s.
\]
Then $H(0,x')=0$. The corner compatibility condition is used here to show
\[
\partial_tH(0,x')=b(0,x')=0.
\]
Indeed, the mixed derivatives of $u$ in \eqref{eq:mixed-derivative-estimate} and the second derivatives of $U$ extend continuously to $t=0$. Their limits are zero, as follows from the zero initial values by testing the spatial derivative identities against compactly supported smooth functions.
Hence $\widetilde f(0,0,x')=f(0,0,x')$, and the half-space corner condition $f(0,\cdot)|_{\partial\R_+^n}=0$ yields $\widetilde f(0,0,x')=0$, and therefore $b(0,x')=0$.

Now $W:=V-H$ satisfies
\[
\partial_tW=\partial_{11}W
\qquad\text{in }Q_T^+,
\qquad
W(t,0,x')=-H(t,x').
\]
The boundary datum $h=-H$ has derivative $h'=-b\in C^{\alpha/2}([0,T];L_\omega^\gamma)$ for each fixed $x'$, by the estimate above. Since $-H(0,x')=0$ and $-\partial_tH(0,x')=0$, Lemma~\ref{lem:half-line-boundary} applies pointwise in $x'$.

For the tangential increments, \eqref{eq:half-line-stability} and $D_{11}V=D_{11}W$ give
\[
\|D_{11}V(t,x_1,x'+h)-D_{11}V(t,x_1,x')\|_{L_\omega^\gamma}
\le C\sup_{s\in[0,T]}\|b(s,x'+h)-b(s,x')\|_{L_\omega^\gamma}
\le C|h|^\alpha|b|_{\alpha;\Sigma_T}^{L_\omega^\gamma}.
\]
Together with the normal and temporal estimates and $\partial_tV=D_{11}V+b$, this proves
\begin{equation}\label{eq:V-estimate}
|\partial_tV|_{(\alpha,\alpha/2);Q_T^+}^{L_\omega^\gamma}
+|D_{11}V|_{(\alpha,\alpha/2);Q_T^+}^{L_\omega^\gamma}
\le C\,|\widetilde f|_{(\alpha,\alpha/2);\Sigma_T}^{L_\omega^\gamma}.
\end{equation}
Combining \eqref{eq:V-estimate} with \eqref{eq:tildef-boundary-bound}, we obtain
\begin{equation}\label{eq:V-estimate-2}
|\partial_tV|_{(\alpha,\alpha/2);Q_T^+}^{L_\omega^\gamma}
+|D_{11}V|_{(\alpha,\alpha/2);Q_T^+}^{L_\omega^\gamma}
\le C\Bigl(
|u|_{0;Q_T^+}^{L_\omega^\gamma}
+|f|_{\alpha;Q_T^+}^{L_\omega^\gamma}
+|f|_{(\alpha,\alpha/2);\Sigma_T}^{L_\omega^\gamma}
\Bigr).
\end{equation}

Let
\[
w:=\widetilde u-V.
\]
By construction, $V$ removes precisely the boundary trace of the reduced forcing. Therefore $w$ satisfies
\begin{equation}\label{eq:w-equation}
\partial_t w=\bar a^{11}(t)D_{11}w+F
\qquad\text{in }Q_T^+,
\end{equation}
with zero boundary and initial conditions, where
\begin{equation}\label{eq:F-equation}
F:=\bigl(\bar a^{11}(t)-1\bigr)D_{11}V
+\widetilde f(t,x)-\bar a^{11}(t)^{-1}\widetilde f(t,0,x').
\end{equation}
Since $D_{11}V=-b$ on $\Sigma_T$ and $\bar a^{11}b=\widetilde f|_{\Sigma_T}$, we have
\begin{equation}\label{eq:F-boundary-zero}
F(t,0,x')=0.
\end{equation}
By \eqref{eq:tildef-alpha-bound} and \eqref{eq:V-estimate-2}, and since the last term in \eqref{eq:F-equation} is independent of $x_1$, we have
\begin{equation}\label{eq:F-bound}
|F|_{\alpha;Q_T^+}^{L_\omega^\gamma}
\le C\Bigl(
|u|_{0;Q_T^+}^{L_\omega^\gamma}
+|f|_{\alpha;Q_T^+}^{L_\omega^\gamma}
+|f|_{(\alpha,\alpha/2);\Sigma_T}^{L_\omega^\gamma}
\Bigr).
\end{equation}
Since $F$ vanishes on $\Sigma_T$, its odd extension is $C^\alpha$ in $x_1$, uniformly in $x'$. Extending $w$ oddly and applying the one-dimensional $L_\omega^\gamma$-valued estimate of~\cite[Corollary~4.2]{DuLiu2019} with $x'$ fixed, we obtain
\begin{equation}\label{eq:w-estimate}
\sup_{x'\in\R^{n-1}}
|D_{11}w(\cdot,\cdot,x')|_{(\alpha,\alpha/2);\R_+\times(0,T]}^{L_\omega^\gamma}
\le C\Bigl(
|w|_{0;Q_T^+}^{L_\omega^\gamma}
+|F|_{\alpha;Q_T^+}^{L_\omega^\gamma}
\Bigr).
\end{equation}
The heat kernel representation of \eqref{eq:V-equation} gives
$|V|_0^{L_\omega^\gamma}\le T|b|_0^{L_\omega^\gamma}\le|b|_0^{L_\omega^\gamma}$.
Since $u=w+V+U$, estimates \eqref{eq:U-estimate}, \eqref{eq:V-estimate-2}, and \eqref{eq:F-bound} therefore yield
\[
\sup_{x'\in\R^{n-1}}
|D_{11}u(\cdot,\cdot,x')|_{(\alpha,\alpha/2);\R_+\times(0,T]}^{L_\omega^\gamma}
\le C\Bigl(
|u|_{0;Q_T^+}^{L_\omega^\gamma}
+|f|_{\alpha;Q_T^+}^{L_\omega^\gamma}
+|f|_{(\alpha,\alpha/2);\Sigma_T}^{L_\omega^\gamma}
\Bigr).
\]

For $n\ge2$, the tangential estimate follows from this normal estimate and \eqref{eq:mixed-derivative-estimate}. Fix $t$ and suppress it in the next identity. For $h\in\R^{n-1}\setminus\{0\}$, set $\delta=|h|$. The fundamental theorem of calculus gives
\[
\begin{aligned}
&\frac1\delta\int_0^\delta
 \bigl(D_{11}u(x_1+r,x'+h)-D_{11}u(x_1+r,x')\bigr)\md r\\
&\quad=\frac1\delta\int_0^1 h\cdot
 \bigl(D_1D_{x'}u(x_1+\delta,x'+\theta h)
       -D_1D_{x'}u(x_1,x'+\theta h)\bigr)\md\theta.
\end{aligned}
\]
Comparing the left side with its integrand at $r=0$ and estimating the right side by the normal H\"older continuity of $D_1D_{x'}u$, we obtain
\[
\begin{aligned}
&\|D_{11}u(t,x_1,x'+h)-D_{11}u(t,x_1,x')\|_{L_\omega^\gamma}\\
&\quad\le C|h|^\alpha\left(
 \sup_{s,z'}[D_{11}u(s,\cdot,z')]_{\alpha;\R_+}^{L_\omega^\gamma}
 +[D_1D_{x'}u]_{\alpha;Q_T^+}^{L_\omega^\gamma}\right).
\end{aligned}
\]
The forward normal interval stays in the half-space, so this also holds at $x_1=0$. Combining the tangential, normal, and temporal estimates gives, also for $n=1$,
\begin{equation}\label{eq:D11-estimate-g-zero}
|D_{11}u|_{(\alpha,\alpha/2);Q_T^+}^{L_\omega^\gamma}
\le C\Bigl(
|u|_{0;Q_T^+}^{L_\omega^\gamma}
+|f|_{\alpha;Q_T^+}^{L_\omega^\gamma}
+|f|_{(\alpha,\alpha/2);\Sigma_T}^{L_\omega^\gamma}
\Bigr).
\end{equation}
Together with \eqref{eq:mixed-derivative-estimate}, this proves \eqref{eq:model-g-zero-estimate} for $T\le1$.

For $T>1$, choose on each interval $(s-1,s]$ a deterministic time cut-off $\chi$ that vanishes near $s-1$, equals one on $[s-1/4,s]$, and has uniformly bounded first derivative. The equation for $\chi u$ has forcing $\chi f+\chi'u$ and boundary trace $\chi f$. Estimate \eqref{eq:first-derivative-model-g-zero} controls $|u|_\alpha^{L_\omega^\gamma}$ by $C(|u|_0^{L_\omega^\gamma}+|f|_\alpha^{L_\omega^\gamma})$, so the cut-off data satisfy the same assumptions and bounds. Apply the short-time estimate, treating the initial interval directly. These intervals control time increments of length at most $1/4$; larger increments follow from the supremum bound. Thus the constant is independent of $T$.
\end{proof}

\begin{proof}[Proof of Proposition~\ref{prop:model-main}]
In this proof, $C$ may depend on $n$, $\alpha$, $\gamma$, $\kappa$, $K$, and $A_\alpha$. We first assume $T\le1$ and then use unit time intervals to obtain the general estimate.
Choose $M>0$ sufficiently large, depending only on the structural constants, so that the pair $(M\delta^{ij},\bar\sigma^{ik})$ satisfies the stochastic parabolicity condition. Let $U$ be the solution of the auxiliary equation
\begin{equation}\label{eq:additive-noise-reduction}
\begin{cases}
dU=M\Delta U\md t+\bigl(\bar\sigma^{ik}(t)D_iU+g^k\bigr)\md w_t^k &\text{in }Q_T^+,\\
U=0 &\text{on }\Sigma_T,\\
U(0,\cdot)=0 &\text{in }\R^n_+.
\end{cases}
\end{equation}
Since $\bar\sigma^{1k}=0$ and $g=0$ on $\Sigma_T$, the stochastic forcing vanishes on $\Sigma_T$: for $i\ge2$, $D_iU=0$ on $\{y_1=0\}$ because $U\equiv0$ there; for $i=1$, $\bar\sigma^{1k}=0$. Hence we may extend all data oddly across $\{x_1=0\}$ and apply the whole-space Schauder estimate for equations with gradient noise (see~\cite[Theorem~1.3]{DuLiu2019}) to obtain
\begin{equation}\label{eq:additive-noise-estimate}
|U|_{(2+\alpha,\alpha/2);Q_T^+}^{L_\omega^\gamma}
\le C|g|_{1+\alpha;Q_T^+}^{L_\omega^\gamma(\elltwo)},
\end{equation}
where $C$ depends only on $n$, $\alpha$, $\gamma$, $\kappa$, $K$, and $M$.

Set $\widehat u:=u-U$. Subtracting the equation for $U$ from \eqref{eq:model-half-space} and collecting the stochastic gradient terms, we obtain
\[
d\widehat u
=\bigl(\bar a^{ij}(t)D_{ij}\widehat u+\widehat f\bigr)\md t
+\bar\sigma^{ik}(t)D_i\widehat u\md w_t^k
\quad\text{in }Q_T^+,
\]
with zero boundary and initial conditions, where
\[
\widehat f:=f+\sum_{i,j=1}^n\bigl(\bar a^{ij}(t)-M\delta^{ij}\bigr)D_{ij}U.
\]
By \eqref{eq:additive-noise-estimate}, the temporal bound $A_\alpha$, and the H\"older product rule,
\begin{align*}
|\widehat f|_{\alpha;Q_T^+}^{L_\omega^\gamma}
+|\widehat f|_{(\alpha,\alpha/2);\Sigma_T}^{L_\omega^\gamma}
\le C\Bigl(
|f|_{\alpha;Q_T^+}^{L_\omega^\gamma}
+|f|_{(\alpha,\alpha/2);\Sigma_T}^{L_\omega^\gamma}
+|g|_{1+\alpha;Q_T^+}^{L_\omega^\gamma(\elltwo)}
\Bigr).
\end{align*}
Since $U(0,\cdot)=0$ and $U$ has the regularity established in \eqref{eq:additive-noise-estimate}, its spatial derivatives vanish at $t=0$. Hence
\[
|\widehat f(0,\cdot)|_{0;\partial\R_+^n}^{L_\omega^\gamma}
=|f(0,\cdot)|_{0;\partial\R_+^n}^{L_\omega^\gamma}=0,
\]
so $\widehat f$ satisfies the assumptions of Lemma~\ref{lem:model-g-zero}. Applying that lemma to $\widehat u$ yields
\begin{align*}
|D_x^2\widehat u|_{(\alpha,\alpha/2);Q_T^+}^{L_\omega^\gamma}
\le C\Bigl(
|\widehat u|_{0;Q_T^+}^{L_\omega^\gamma}
+|\widehat f|_{\alpha;Q_T^+}^{L_\omega^\gamma}
+|\widehat f|_{(\alpha,\alpha/2);\Sigma_T}^{L_\omega^\gamma}
\Bigr).
\end{align*}
Combining the estimate for $\widehat f$ with the above and observing $u=\widehat u+U$, we conclude using \eqref{eq:additive-noise-estimate} that
\begin{align*}
|D_x^2u|_{(\alpha,\alpha/2);Q_T^+}^{L_\omega^\gamma}
\le C\Bigl(
|u|_{0;Q_T^+}^{L_\omega^\gamma}
+|f|_{\alpha;Q_T^+}^{L_\omega^\gamma}
+|f|_{(\alpha,\alpha/2);\Sigma_T}^{L_\omega^\gamma}
+|g|_{1+\alpha;Q_T^+}^{L_\omega^\gamma(\elltwo)}
\Bigr),
\end{align*}
which proves \eqref{eq:model-main-estimate} for $T\le1$.

For arbitrary $T$, we first obtain
\[
 |Du|_{0;Q_T^+}^{L_\omega^\gamma}
 \le C\left(|u|_{0;Q_T^+}^{L_\omega^\gamma}
       +|f|_{\alpha;Q_T^+}^{L_\omega^\gamma}
       +|g|_{1+\alpha;Q_T^+}^{L_\omega^\gamma(\elltwo)}\right).
\]
Indeed, on each unit time interval let $U$ solve \eqref{eq:additive-noise-reduction} with zero value at the left endpoint. By \eqref{eq:additive-noise-estimate}, the equation for $u-U$ has no additive noise and its forcing $f+(\bar a^{ij}-M\delta^{ij})D_{ij}U$ has the required spatial bound. The local energy approximation proving \eqref{eq:first-derivative-model-g-zero}, applied on cylinders of radius $1/2$ ending in the last quarter of the interval, gives the displayed estimate. These cylinders stay away from the left endpoint; on the initial truncated interval, both $u$ and $U$ have zero initial value.

This also controls $|u|_\alpha^{L_\omega^\gamma}$. Use the time cut-offs from the preceding proof: the data for $\chi u$ are $\chi f+\chi'u$ and $\chi g$, with drift trace $\chi f$. Their norms are uniformly bounded by the right-hand side of \eqref{eq:model-main-estimate}. The short-time estimate and the same treatment of time increments prove the proposition.
\end{proof}

\section{Proof of Theorem~\ref{thm:schauder}}\label{sec:global-estimate}

We now pass from the flat-boundary model studied in Section~\ref{sec:halfspace} to general domains. 
The geometric part of the argument is standard and follows the boundary flattening procedure used in~\cite[Definition~2.1 and Section~3.2]{Du2020}. The interior estimate is exactly the same as in the whole-space Schauder theory of~\cite[Sections~4 and~5]{DuLiu2019}. Thus the only points that require verification are the following: after flattening the boundary, the transformed equation still satisfies the flat-boundary compatibility condition required in Section~\ref{sec:halfspace}; and, after localization, the variable coefficients may indeed be treated as a perturbation of a frozen model equation.

Fix $z\in\partial G$. Choose $r_z>0$ and a $C^{2+\alpha}$-diffeomorphism
\[
\Psi_z:B_{r_z}(z)\to U_z\subset\R^n
\]
such that $\Psi_z(z)=0$ and
\[
\Psi_z\bigl(B_{r_z}(z)\cap G\bigr)=U_z\cap\R^n_+,
\qquad
\Psi_z\bigl(B_{r_z}(z)\cap\partial G\bigr)=U_z\cap\{y_1=0\}.
\]
Write $\Phi_z:=\Psi_z^{-1}$ and, for $y\in U_z\cap\R^n_+$,
\[
\widetilde u(t,y):=u\bigl(t,\Phi_z(y)\bigr).
\]
A direct computation shows that $\widetilde u$ satisfies, in a smaller neighborhood of the origin,
\begin{equation}\label{eq:flattened-equation}
 d\widetilde u
 =\bigl(\widetilde a^{ij}D_{ij}\widetilde u+\widetilde b^iD_i\widetilde u+\widetilde c\widetilde u+\widetilde f\bigr)dt
 +\bigl(\widetilde\sigma^{ik}D_i\widetilde u+\widetilde\nu^k\widetilde u+\widetilde g^k\bigr)\md w_t^k,
\end{equation}
with zero boundary condition on $\{y_1=0\}$. Here the transformed coefficients are of the same type as in \eqref{eq:main}; moreover, by the chain rule and the regularity of $\Psi_z$ and $\Phi_z$, they satisfy the same H\"older assumptions as the original coefficients, with constants depending only on the structural data and on the $C^{2+\alpha}$-character of $G$.

The compatibility condition is preserved under this change of variables. Indeed, if $y=\Psi_z(x)$, then
\[
D_i u(t,x)=D_\ell \widetilde u(t,y) D_i\Psi_z^\ell(x),
\]
so the transformed gradient coefficient is given by
\[
\widetilde\sigma^{\ell k}(t,y)=\sigma^{ik}(t,x)D_i\Psi_z^\ell(x).
\]
On the boundary, the first component $\Psi_z^1$ is a defining function of $\partial G\cap B_{r_z}(z)$, and hence $D\Psi_z^1(x)$ is normal to the boundary and therefore parallel to $n(x)$. Consequently,
\[
\widetilde\sigma^{1k}(t,y)=\sigma^{ik}(t,x)D_i\Psi_z^1(x)=0
\qquad\text{whenever } y_1=0,
\]
by Assumption~\ref{ass:compatibility}.

Since the boundary flattening map is independent of time, Assumption~\ref{ass:coefficients} implies that the frozen coefficients $\bar a^{ij}(t):=\widetilde a^{ij}(t,0)$ satisfy the temporal H\"older condition of Proposition~\ref{prop:model-main}, with $A_\alpha\le CL_\alpha$. Here $C$ depends only on the $C^{2+\alpha}$-character of $G$.

We next derive a local boundary estimate. Let $\chi\in C_0^\infty(B_{2r}(0))$ be such that $\chi\equiv1$ on $B_r(0)$, where $r>0$ will be chosen later, and set
\[
v:=\chi\widetilde u.
\]
After extending the transformed coefficients from $U_z\cap\R^n_+$ to the whole half-space in a standard way, we may write the equation for $v$ on $Q_T^+=\R^n_+\times(0,T]$ in the form
\begin{equation}\label{eq:localized-flattened-equation}
 dv=\bigl(\widetilde a^{ij}(t,0)D_{ij}v+F\bigr)dt
 +\bigl(\widetilde\sigma^{ik}(t,0)D_iv+G^k\bigr)\md w_t^k,
\end{equation}
with zero boundary and initial conditions. Here
\begin{align*}
F={}&\chi\widetilde f
+\bigl(\widetilde a^{ij}(t,y)-\widetilde a^{ij}(t,0)\bigr)D_{ij}v
+\widetilde b^iD_iv+\widetilde c v \\
&\quad -2\widetilde a^{ij}D_i\chi\,D_j\widetilde u
-\widetilde a^{ij}D_{ij}\chi\,\widetilde u
-\widetilde b^iD_i\chi\,\widetilde u,
\end{align*}
and
\begin{align*}
G^k={}\chi\widetilde g^k
+\bigl(\widetilde\sigma^{ik}(t,y)-\widetilde\sigma^{ik}(t,0)\bigr)D_iv
+\widetilde\nu^k v
-\widetilde\sigma^{ik}(t,y)D_i\chi\,\widetilde u.
\end{align*}
On $\Sigma_T$ we have $v=0$ and $\widetilde g=0$. We verify that every term in $G$ vanishes on $\Sigma_T$:
\begin{itemize}
\item $\chi\widetilde g^k=0$ on $\Sigma_T$ because $\widetilde g=0$ on $\Sigma_T$;
\item $\widetilde\nu^k v=0$ on $\Sigma_T$ because $v=0$ there;
\item $-\widetilde\sigma^{ik}(t,y)D_i\chi\,\widetilde u=0$ on $\Sigma_T$ because $\widetilde u=0$ there;
\item for $i\ge2$, $D_i v=0$ on $\Sigma_T$ (as a tangential derivative of $v\equiv0$ on $\{y_1=0\}$), hence $(\widetilde\sigma^{ik}(t,y)-\widetilde\sigma^{ik}(t,0))D_i v=0$ there;
\item for $i=1$, the transformed compatibility condition gives $\widetilde\sigma^{1k}(t,y)=0$ on $\{y_1=0\}$, and the frozen coefficient satisfies $\widetilde\sigma^{1k}(t,0)=0$. Hence $(\widetilde\sigma^{1k}(t,y)-\widetilde\sigma^{1k}(t,0))D_1 v=0$ on $\Sigma_T$.
\end{itemize}
Therefore every term in $G$ vanishes on $\Sigma_T$, and hence
\[
G=0\qquad\text{on }\Sigma_T.
\]

The corner compatibility condition is inherited by the localized forcing $F$. Indeed, at $t=0$, the localized solution $v$ and its spatial derivatives vanish because $v(0,\cdot)=0$. Hence all commutator and coefficient-oscillation terms vanish at $\{0\}\times\partial\R_+^n$. The only remaining boundary contribution is $\chi\widetilde f(0,\cdot)$, which is zero by $f(0,\cdot)=0$ on $\partial G$. Thus $F(0,\cdot)=0$ on $\partial\R_+^n$, as required by Proposition~\ref{prop:model-main}.

Applying Proposition~\ref{prop:model-main}, we obtain
\begin{align}
|v|_{(2+\alpha,\alpha/2);Q_T^+}^{L_\omega^\gamma}
\le C\Bigl(
|v|_{0;Q_T^+}^{L_\omega^\gamma}
+|F|_{\alpha;Q_T^+}^{L_\omega^\gamma}
+|F|_{(\alpha,\alpha/2);\Sigma_T}^{L_\omega^\gamma}
+|G|_{1+\alpha;Q_T^+}^{L_\omega^\gamma(\elltwo)}
\Bigr).
\label{eq:flat-local-estimate}
\end{align}
We estimate the coefficient oscillations also in the boundary parabolic norm.
The H\"older product rule gives
\begin{align*}
&\bigl|\bigl(\widetilde a^{ij}(t,\cdot)-\widetilde a^{ij}(t,0)\bigr)D_{ij}v\bigr|_{\alpha;Q_T^+}^{L_\omega^\gamma}
+\bigl|\bigl(\widetilde a^{ij}(t,\cdot)-\widetilde a^{ij}(t,0)\bigr)D_{ij}v\bigr|_{(\alpha,\alpha/2);\Sigma_T}^{L_\omega^\gamma}\\
&\quad
+\bigl|\bigl(\widetilde\sigma^{ik}(t,\cdot)-\widetilde\sigma^{ik}(t,0)\bigr)D_iv\bigr|_{1+\alpha;Q_T^+}^{L_\omega^\gamma(\elltwo)}\\
&\qquad\le Cr^\alpha|v|_{(2+\alpha,\alpha/2);Q_T^+}^{L_\omega^\gamma}
+C|v|_{2;Q_T^+}^{L_\omega^\gamma}.
\end{align*}
Here the temporal H\"older seminorm of the coefficient difference need not
be small: in the product estimate it multiplies
$|D^2v|_{0;Q_T^+}^{L_\omega^\gamma}$ and is controlled by $L_\alpha$.
The spatial interpolation inequality
\[
|v|_{2;Q_T^+}^{L_\omega^\gamma}
\le\varepsilon|v|_{(2+\alpha,\alpha/2);Q_T^+}^{L_\omega^\gamma}
+C_\varepsilon|v|_{0;Q_T^+}^{L_\omega^\gamma}
\]
therefore absorbs these terms in \eqref{eq:flat-local-estimate}, after
choosing $r$ and then $\varepsilon$ sufficiently small.

The lower-order terms require control of the temporal increments of
$\widetilde u$ and $D\widetilde u$.  The equation and the
Burkholder--Davis--Gundy inequality give, for $0<t-s\le1$,
\[
\|\widetilde u(t)-\widetilde u(s)\|_{0;B_R^+}^{L_\omega^\gamma}
\le C|t-s|^{1/2}\Bigl(
 |\widetilde u|_{2;B_R^+\times(0,T]}^{L_\omega^\gamma}
 +|\widetilde f|_{0;B_R^+\times(0,T]}^{L_\omega^\gamma}
 +|\widetilde g|_{0;B_R^+\times(0,T]}^{L_\omega^\gamma(\elltwo)}\Bigr).
\]
For larger time increments we use the supremum norm.
Applying local spatial interpolation both to $\widetilde u$ and to
$\widetilde u(t)-\widetilde u(s)$ consequently gives, for
$r\le\rho<R\le2r$ and every $\varepsilon>0$,
\begin{align*}
&|\widetilde u|_{1+\alpha;B_\rho^+\times(0,T]}^{L_\omega^\gamma}
 +|\widetilde u|_{(\alpha,\alpha/2);B_\rho^+\times(0,T]}^{L_\omega^\gamma}
 +|D\widetilde u|_{(\alpha,\alpha/2);B_\rho^+\times(0,T]}^{L_\omega^\gamma}\\
&\quad\le\varepsilon
 |\widetilde u|_{(2+\alpha,\alpha/2);B_R^+\times(0,T]}^{L_\omega^\gamma}
 +\frac{C_\varepsilon}{(R-\rho)^p}\Bigl(
 |\widetilde u|_{0;B_R^+\times(0,T]}^{L_\omega^\gamma}
 +|\widetilde f|_{0;B_R^+\times(0,T]}^{L_\omega^\gamma}
 +|\widetilde g|_{0;B_R^+\times(0,T]}^{L_\omega^\gamma(\elltwo)}\Bigr),
\end{align*}
for some fixed $p>0$.  To obtain the small coefficient in the temporal
seminorm of $D\widetilde u$, one first interpolates the spatial derivative
of $\widetilde u(t)-\widetilde u(s)$ between its zeroth and second
spatial derivatives, and then interpolates
$|\widetilde u|_2$ in the preceding time-increment estimate.
The constants are independent of $T$.

We now use cut-offs between two nested radii, with an intermediate radius
for the preceding interpolation estimate.  On the flat boundary, all
terms containing $\widetilde u$ or $v$ without a derivative vanish.
The remaining lower-order terms in $F$ contain only
$\widetilde b^iD_iv$ and
$\widetilde a^{ij}D_i\chi D_j\widetilde u$; their boundary temporal norms
are controlled by the assumptions on $\widetilde a,\widetilde b$ and the
preceding estimate.  The spatial norms of all commutators are controlled
by the same estimate and the coefficient bounds.  In particular,
$|\widetilde\nu v|_{1+\alpha}^{L_\omega^\gamma(\elltwo)}
\le C|v|_{1+\alpha}^{L_\omega^\gamma}$.
Thus \eqref{eq:flat-local-estimate} implies, for every $\varepsilon>0$,
\begin{align*}
&|\widetilde u|_{(2+\alpha,\alpha/2);B_\rho^+\times(0,T]}^{L_\omega^\gamma}\\
&\quad\le\varepsilon
 |\widetilde u|_{(2+\alpha,\alpha/2);B_R^+\times(0,T]}^{L_\omega^\gamma}
 +\frac{C_\varepsilon}{(R-\rho)^p}\Bigl(
 |\widetilde u|_{0;B_R^+\times(0,T]}^{L_\omega^\gamma}
 +|\widetilde f|_{\alpha;B_R^+\times(0,T]}^{L_\omega^\gamma}\\
&\hspace{42mm}
 +|\widetilde f|_{(\alpha,\alpha/2);(B_R\cap\{y_1=0\})\times[0,T]}^{L_\omega^\gamma}
 +|\widetilde g|_{1+\alpha;B_R^+\times(0,T]}^{L_\omega^\gamma(\elltwo)}\Bigr),
\end{align*}
with $p$ increased if necessary.
Since the solution has finite $C_{x,t}^{2+\alpha,\alpha/2}$ norm,
the usual iteration on radii increasing from $r$ to $2r$, with
$\varepsilon<2^{-p}$, gives
\begin{align}
&|\widetilde u|_{(2+\alpha,\alpha/2);B_r^+\times(0,T]}^{L_\omega^\gamma}
\notag\\
&\quad\le C\Bigl(
 |\widetilde u|_{0;B_{2r}^+\times(0,T]}^{L_\omega^\gamma}
 +|\widetilde f|_{\alpha;B_{2r}^+\times(0,T]}^{L_\omega^\gamma}
\notag\\
&\hspace{23mm}
 +|\widetilde f|_{(\alpha,\alpha/2);(B_{2r}\cap\{y_1=0\})\times[0,T]}^{L_\omega^\gamma}
 +|\widetilde g|_{1+\alpha;B_{2r}^+\times(0,T]}^{L_\omega^\gamma(\elltwo)}\Bigr).
\label{eq:flattened-local-estimate}
\end{align}
The bounds on the lower-order temporal increments also control the
corresponding chain-rule terms when the coordinates are transformed
back.  The same nested-radius argument in the original coordinate
patch therefore yields the following local boundary estimate.

\begin{lemma}\label{lem:local-boundary}
There exist $r_0>0$ and $C>0$, depending only on
\[
n,\ \alpha,\ \gamma,\ \kappa,\ K,\ L,\ L_\alpha,
\ \text{and the }C^{2+\alpha}\text{-character of }G,
\]
such that for every $z\in\partial G$ and every quasi-classical solution $u\in C_{x,t}^{2+\alpha,\alpha/2}(Q_T;L_\omega^\gamma)$ of \eqref{eq:main}--\eqref{eq:boundary-initial},
\begin{align}
|u|_{(2+\alpha,\alpha/2);(B_{r_0}(z)\cap G)\times(0,T]}^{L_\omega^\gamma}
\le C\Bigl(
|u|_{0;(B_{2r_0}(z)\cap G)\times(0,T]}^{L_\omega^\gamma}
+|f|_{\alpha;(B_{2r_0}(z)\cap G)\times(0,T]}^{L_\omega^\gamma}
\notag\\
\qquad
+|f|_{(\alpha,\alpha/2);\Gamma_T\cap(B_{2r_0}(z)\times[0,T])}^{L_\omega^\gamma}
+|g|_{1+\alpha;(B_{2r_0}(z)\cap G)\times(0,T]}^{L_\omega^\gamma(\elltwo)}
\Bigr).
\label{eq:local-boundary-estimate-new}
\end{align}
\end{lemma}

For interior points there is nothing new. If $B_{2r}(x_0)\Subset G$, then the interior estimate in~\cite[Section~4]{DuLiu2019}, together with the localization argument in~\cite[Section~5]{DuLiu2019}, yields
\begin{align}
|u|_{(2+\alpha,\alpha/2);B_r(x_0)\times(0,T]}^{L_\omega^\gamma}
\le C\Bigl(
|u|_{0;B_{2r}(x_0)\times(0,T]}^{L_\omega^\gamma}
+|f|_{\alpha;B_{2r}(x_0)\times(0,T]}^{L_\omega^\gamma}
+|g|_{1+\alpha;B_{2r}(x_0)\times(0,T]}^{L_\omega^\gamma(\elltwo)}
\Bigr),
\label{eq:local-interior-estimate-new}
\end{align}
where $C$ depends only on the same structural quantities.

We now globalize. The local boundary estimate (Lemma~\ref{lem:local-boundary}) and the interior estimate \eqref{eq:local-interior-estimate-new} hold on each coordinate patch with constants depending only on the $C^{2+\alpha}$-character of $G$. Covering $\bar G$ by a locally finite family of such patches and taking the supremum over all patches, we obtain
\begin{align}
|u|_{(2+\alpha,\alpha/2);Q_T}^{L_\omega^\gamma}
\le C\Bigl(
|u|_{0;Q_T}^{L_\omega^\gamma}
+|f|_{\alpha;Q_T}^{L_\omega^\gamma}
+|f|_{(\alpha,\alpha/2);\Gamma_T}^{L_\omega^\gamma}
+|g|_{1+\alpha;Q_T}^{L_\omega^\gamma(\elltwo)}
\Bigr).
\label{eq:global-a-priori-before-gronwall}
\end{align}
Finally, let $\mathcal D$ denote the sum of the three data norms on the right-hand side of \eqref{eq:schauder-estimate}. The equation, the Burkholder--Davis--Gundy inequality, and the weighted Cauchy--Schwarz inequality give, for $0<t\le T$,
\[
\bigl(|u|_{0;Q_t}^{L_\omega^\gamma}\bigr)^2
\le Ce^t\int_0^t e^{-s}
\Bigl(\bigl(|u|_{(2+\alpha,\alpha/2);Q_s}^{L_\omega^\gamma}\bigr)^2+\mathcal D^2\Bigr)\md s.
\]
Apply \eqref{eq:global-a-priori-before-gronwall} on $Q_s$, whose constant is independent of $s$, and then apply Gronwall's inequality to $e^{-t}(|u|_{0;Q_t}^{L_\omega^\gamma})^2$. This gives
\[
|u|_{0;Q_T}^{L_\omega^\gamma}
\le Ce^{CT}\Bigl(
|f|_{\alpha;Q_T}^{L_\omega^\gamma}
+|f|_{(\alpha,\alpha/2);\Gamma_T}^{L_\omega^\gamma}
+|g|_{1+\alpha;Q_T}^{L_\omega^\gamma(\elltwo)}
\Bigr),
\]
which, together with \eqref{eq:global-a-priori-before-gronwall}, yields \eqref{eq:schauder-estimate}. The proof of Theorem~\ref{thm:schauder} is complete.

\section{Proof of Theorem~\ref{thm:solvability}}\label{sec:solvability}

By Theorem~\ref{thm:schauder}, any quasi-classical solution in $C_{x,t}^{2+\alpha,\alpha/2}(Q_T;L_\omega^\gamma)$ of \eqref{eq:main}--\eqref{eq:boundary-initial}
satisfies the a priori estimate
\begin{equation}\label{eq:section5-apriori}
	|u|_{(2+\alpha,\alpha/2);Q_T}^{L_\omega^\gamma}
	\le
	C e^{CT}
	\Bigl(
	|f|_{\alpha;Q_T}^{L_\omega^\gamma}
	+
	|f|_{(\alpha,\alpha/2);\Gamma_T}^{L_\omega^\gamma}
	+
	|g|_{1+\alpha;Q_T}^{L_\omega^\gamma(\elltwo)}
	\Bigr).
\end{equation}
In particular, uniqueness follows immediately by applying \eqref{eq:section5-apriori} to the
difference of two solutions. It remains to prove existence.

For $s\in[0,1]$, consider the family of equations
\begin{equation}\label{eq:continuity-family}
	du=(L_su+f)\md t+(\Lambda_s^ku+g^k)\md w_t^k
	\quad\text{in }Q_T,
\end{equation}
with the boundary and initial conditions
\[
u=0 \quad\text{on }\Gamma_T,
\qquad
u(0,\cdot)=0 \quad\text{in }G,
\]
where
\[
L_s:=s\bigl(a^{ij}D_{ij}+b^iD_i+c\bigr)+(1-s)\Delta,
\qquad
\Lambda_s^k:=s\bigl(\sigma^{ik}D_i+\nu^k\bigr).
\]
Since Assumptions~\ref{ass:compatibility}--\ref{ass:data} are stable under this interpolation, the estimate
\eqref{eq:section5-apriori} holds for \eqref{eq:continuity-family} with a constant independent of
$s\in[0,1]$.

It therefore suffices to verify solvability for one value of $s$, and then apply the standard
method of continuity. We choose $s=0$, for which \eqref{eq:continuity-family} becomes
\begin{equation}\label{eq:base-equation}
	du=(\Delta u+f)\md t+g^k\md w_t^k
	\quad\text{in }Q_T,
	\qquad
	u=0 \text{ on }\Gamma_T,
	\qquad
	u(0,\cdot)=0 \text{ in }G.
\end{equation}

We construct the solution of \eqref{eq:base-equation} by approximation at a lower H\"older exponent, keeping uniform estimates at the original exponent. For $0<\eta\le\alpha$, write
\[
\mathcal N_\eta(F,H)
:=|F|_{\eta;Q_T}^{L_\omega^\gamma}
 +|\operatorname{Tr}F|_{(\eta,\eta/2);\Gamma_T}^{L_\omega^\gamma}
 +|H|_{1+\eta;Q_T}^{L_\omega^\gamma(\elltwo)},
\]
where $\operatorname{Tr}$ denotes the boundary trace.

Using uniformly controlled $C^{2+\alpha}$ boundary charts and a uniformly locally finite partition of unity, let $S_\varepsilon$ be spatial extension followed by convolution and restriction to $G$, and let $J_\varepsilon$ be tangential regularization on $\partial G$. Choose a right inverse $\mathcal L$ of the trace by constant extension in the flattened normal coordinate and a fixed cut-off. Then $\mathcal L$ is bounded in the spatial $C^r$ norms used below, for $0\le r\le2+\alpha$. These deterministic spatial operators act on the Banach-valued data and preserve predictability.

Put $\varphi=\operatorname{Tr}f$ and define
\[
\begin{aligned}
 f_\varepsilon
 &=S_\varepsilon f-\mathcal L\bigl(\operatorname{Tr}S_\varepsilon f-J_\varepsilon\varphi\bigr),\\
 g_\varepsilon
 &=S_\varepsilon g-\mathcal L\bigl(\operatorname{Tr}S_\varepsilon g\bigr).
\end{aligned}
\]
Thus
\[
 \operatorname{Tr}f_\varepsilon=J_\varepsilon\varphi,\qquad
 \operatorname{Tr}g_\varepsilon=0,\qquad
 \operatorname{Tr}f_\varepsilon(0)=0.
\]
For fixed $\varepsilon>0$, these data are spatially $C^{2+\alpha}$. The spatial mollifier estimates give
\[
 \sup_{0<\varepsilon<\varepsilon_0}\mathcal N_\alpha(f_\varepsilon,g_\varepsilon)
 \le C\mathcal N_\alpha(f,g),\qquad
 \mathcal N_\beta(f_\varepsilon-f,g_\varepsilon-g)
 \le C_\beta\varepsilon^{\alpha-\beta}\mathcal N_\alpha(f,g)
 \quad(0<\beta<\alpha).
\]
For the boundary term, interpolate the uniform $C^{\alpha,\alpha/2}$ bound with
$|J_\varepsilon\varphi-\varphi|_0^{L_\omega^\gamma}
\le C\varepsilon^\alpha|\varphi|_{\alpha;\Gamma_T}^{L_\omega^\gamma}$
to obtain the same rate in $C^{\beta,\beta/2}$.

Choose $\varepsilon_m\downarrow0$ and write $(f_m,g_m)=(f_{\varepsilon_m},g_{\varepsilon_m})$. For each $m$, consider
\begin{equation}\label{eq:base-equation-m}
 du_m=(\Delta u_m+f_m)\md t+g_m^k\md w_t^k
 \quad\text{in }Q_T,\qquad
 u_m=0\text{ on }\Gamma_T,\qquad u_m(0,\cdot)=0.
\end{equation}
To construct $u_m$, set
\[
 H_m(t,x)=\int_0^t f_m(s,x)\md s
             +\int_0^t g_m^k(s,x)\md w_s^k.
\]
Spatial differentiation and the Burkholder--Davis--Gundy inequality show that $H_m$ is spatially $C^{2+\alpha}$ with values in $L_\omega^\gamma$, and $D^2H_m$ is $1/2$-H\"older in time. Thus $\Delta H_m\in C_{x,t}^{\alpha,\alpha/2}(\bar Q_T;L_\omega^\gamma)$. Its trace
$\operatorname{Tr}H_m(t)=\int_0^tJ_{\varepsilon_m}\varphi(s)\md s$
has the classical $C^{2+\alpha,1+\alpha/2}$ boundary regularity: for fixed $\varepsilon$, $J_\varepsilon:C^0(\partial G)\to C^{2+\alpha}(\partial G)$ is bounded and preserves temporal $\alpha/2$-H\"older regularity. The corner compatibility follows from $\varphi(0)=0$. The deterministic Dirichlet heat theory with $L_\omega^\gamma$-valued data therefore solves
\[
 \partial_t z_m=\Delta z_m+\Delta H_m,\qquad
 z_m|_{\Gamma_T}=-H_m,\qquad z_m(0,\cdot)=0.
\]
Its causal solution operator preserves predictability, and $u_m=z_m+H_m$ solves \eqref{eq:base-equation-m}.

Fix $0<\beta<\alpha$. The Schauder estimate at exponent $\beta$ gives
\[
 |u_m-u_\ell|_{(2+\beta,\beta/2);Q_T}^{L_\omega^\gamma}
 \le C_\beta e^{C_\beta T}
       \mathcal N_\beta(f_m-f_\ell,g_m-g_\ell)
 \longrightarrow0.
\]
Hence $u_m\to u$ in $C_{x,t}^{2+\beta,\beta/2}$; the spatial integral identities identify the uniformly convergent derivatives through order two with the strong derivatives of $u$. The uniform estimate at exponent $\alpha$ and passage to the limit in the H\"older quotients yield
\[
 |u|_{(2+\alpha,\alpha/2);Q_T}^{L_\omega^\gamma}
 \le\liminf_{m\to\infty}|u_m|_{(2+\alpha,\alpha/2);Q_T}^{L_\omega^\gamma}
 \le Ce^{CT}\mathcal N_\alpha(f,g).
\]

The convergence of $D^2u_m,f_m,g_m$ and the Burkholder--Davis--Gundy inequality imply convergence of the integral equations in $L_\omega^\gamma C([0,T])$, uniformly in $x$. Choose a subsequence with summable successive differences in $\sup_x\|\cdot(x)\|_{L_\omega^\gamma C([0,T])}$. Its pointwise limit, set to zero where absent, is jointly predictable. For each fixed $x$, convergence is uniform in time almost surely, so this representative is continuous in time and satisfies \eqref{eq:base-equation} almost surely for all $t\in[0,T]$. The initial and boundary traces also pass to the limit. Thus $u$ is a quasi-classical solution of the base equation in the required H\"older class.

We finish with the method of continuity; compare~\cite[Section~2.2]{DuLiu2019}. Let $X_0$ consist of predictable fields with zero initial and boundary values and finite norm
\[
 \|v\|_{X_0}
 :=|v|_{(2+\alpha,\alpha/2);\bar Q_T}^{L_\omega^\gamma}
       +|v|_{(\alpha,\alpha/2);\bar Q_T}^{L_\omega^\gamma}.
\]
Spatial derivatives are taken with their continuous extensions to $t=0$, so $Dv(0)=D^2v(0)=0$. This is a Banach space. Spatial interpolation applied to time differences gives
\[
 |Dv|_{(\alpha,\alpha/2);\bar Q_T}^{L_\omega^\gamma}
 \le C\|v\|_{X_0}.
\]
For solutions of \eqref{eq:continuity-family}, the equation and the Burkholder--Davis--Gundy inequality control the temporal $1/2$-H\"older increments of $u$ on intervals of length at most one; larger increments are controlled by $|u|_0$. Thus \eqref{eq:section5-apriori} gives
\[
 \|u\|_{X_0}\le Ce^{CT}\mathcal N_\alpha(f,g),
\]
uniformly in $s$.

Suppose \eqref{eq:continuity-family} is solvable at $s_0$ for every admissible data pair. For $v\in X_0$, define $u=\mathcal R_{s,s_0}v$ to be the solution of
\[
 du=\bigl(L_{s_0}u+(L_s-L_{s_0})v+f\bigr)\md t
       +\bigl(\Lambda_{s_0}^ku+(\Lambda_s^k-\Lambda_{s_0}^k)v+g^k\bigr)\md w_t^k,
\]
with homogeneous initial and boundary conditions. The shifted drift trace consists of $a^{ij}D_{ij}v$, $b^iD_iv$, and $\Delta v$, whose time regularity follows from the coefficient assumptions and the $X_0$ norm. The shifted noise has zero trace by tangentiality of $\sigma$ and the zero trace of $v$. The corner condition follows from $v(0)=Dv(0)=D^2v(0)=0$. Thus the shifted data are admissible, and the product estimates give
\[
 \mathcal N_\alpha\bigl((L_s-L_{s_0})v,
                  (\Lambda_s-\Lambda_{s_0})v\bigr)
 \le C|s-s_0|\|v\|_{X_0}.
\]
The a priori estimate in $X_0$ therefore implies
\[
 \|\mathcal R_{s,s_0}v_1-\mathcal R_{s,s_0}v_2\|_{X_0}
 \le C_T|s-s_0|\|v_1-v_2\|_{X_0},
\]
where $C_T$ is independent of $s,s_0$. The map is a contraction whenever
$|s-s_0|<(2C_T)^{-1}$. Starting from $s=0$, finitely many steps of this uniform size reach $s=1$, completing the proof of Theorem~\ref{thm:solvability}.

\section{Proof of Theorem~\ref{thm:classical}}\label{sec:classical}

Choose a common increasing sequence of stopping times $\tau_m\le T$ for the coefficients and free terms in the definition of the locally bounded classes, with
$\PP\{\tau_m=T\}\to1$. For each $m$, replace every coefficient and free term $h$ by
\[
 h^{(m)}(t,x):=h(t\wedge\tau_m,x).
\]
These fields are predictable: indeed,
$h^{(m)}(t)=h(t)\mathbf1_{\{t\le\tau_m\}}+h(\tau_m)\mathbf1_{\{t>\tau_m\}}$,
and the second term is a left-continuous adapted process. Stopping preserves the spatial bounds, parabolicity, tangentiality, and boundary compatibility conditions. It also preserves the pathwise boundary time seminorms, since $t\mapsto t\wedge\tau_m$ is Lipschitz with constant one. Thus, for every $\theta\in(0,\alpha)$ and every finite $\gamma\ge2$, the stopped fields satisfy Assumptions~\ref{ass:compatibility}--\ref{ass:data} with exponent $\theta$ and deterministic bounds depending on $m$ and $\theta$.

We use uniqueness on stochastic intervals in making this localization. The $W^{2,p}$-theory of~\cite[Theorem~2.5 and the stopping-time formulation immediately following it]{Du2020}, with $p=2$, gives uniqueness up to any stopping time bounded by $T$. Its assumptions hold for the stopped coefficients: $a$ is uniformly continuous in space, $\sigma$ is Lipschitz and tangential, and $b,c,\nu,D\nu$ are bounded. Since $G$ is bounded, a solution with bounded $L_\omega^2$-valued spatial derivatives through order two belongs to the required $L^2(\Omega\times(0,\tau);W^{2,2}(G))$ class. The stochastic term has zero trace and belongs to the corresponding $W^{1,2}$ space. Consequently two such solutions whose equations agree up to $\tau$ agree on that stochastic interval. The same conclusion holds for locally bounded pathwise solutions after intersecting with their norm-localizing stopping times.

Let $u^{(m)}$ be the solution for the stopped data given by Theorem~\ref{thm:solvability}. The uniqueness just described identifies the solutions obtained with different choices of $\theta$ and $\gamma$. Therefore
\[
 u^{(m)}\in C_{x,t}^{2+\theta,\theta/2}(\bar Q_T;L_\omega^\gamma)
 \qquad(0<\theta<\alpha,\ 2\le\gamma<\infty).
\]
To pass to pathwise regularity, fix $0<\beta<\alpha$, choose $\theta\in(\beta,\alpha)$, and take
\begin{equation}\label{eq:kolmogorov-gamma}
 \gamma>\frac{n+2}{\theta-\beta}.
\end{equation}
The temporal estimates for the lower derivatives established in Section~\ref{sec:solvability}, together with the estimate for $D^2u^{(m)}$, allow the anisotropic Kolmogorov theorem to be applied simultaneously to $u^{(m)}$ and its spatial derivatives through order two; compare the discussion after~\cite[Theorem~1.1]{DuLiu2019}. The strong derivative identities identify the resulting modifications as derivatives of the same random field. Since $\bar G$ is compact, a finite collection of boundary and interior charts yields
\[
 u^{(m)}(\omega)\in C_{x,t}^{2+\beta,\beta/2}(\bar Q_T)
 \qquad\text{almost surely}.
\]
Applying this argument to a countable sequence $\beta_k\uparrow\alpha$ gives one predictable modification with this property for every $\beta<\alpha$.

The equations for $u^{(m+1)}$ and $u^{(m)}$ coincide up to $\tau_m$. Stopped uniqueness and continuity of these modifications give
\[
 u^{(m+1)}=u^{(m)}\quad\text{on }\bar G\times[0,\tau_m]
 \qquad\text{almost surely}.
\]
They therefore patch along the stochastic intervals to a predictable solution $u$. Because the events $\{\tau_m=T\}$ increase to an event of probability one, this defines $u$ on all of $\bar Q_T$, including its terminal time. Its paths belong to $C_{x,t}^{2+\beta,\beta/2}(\bar Q_T)$ for every $\beta<\alpha$, and its initial and boundary traces are zero.

It remains to obtain deterministic norm bounds after localization, with a common sequence of stopping times for all exponents. Fix $\beta_k\uparrow\alpha$ and put
\[
 N_k(t):=|u(\omega)|_{(2+\beta_k,\beta_k/2);G\times[0,t]}.
\]
Each $N_k$ is an adapted nondecreasing continuous process starting from zero. Adaptedness follows by taking the suprema over a countable dense set; continuity follows from the pathwise regularity at an exponent strictly larger than $\beta_k$. Choose deterministic numbers $B_k>0$ such that
$\PP\{N_k(T)>B_k\}\le2^{-k}$, and set $c_k=2^{-k}/(1+B_k)$. The Borel--Cantelli lemma gives
\[
 A(T):=\sum_{k=1}^{\infty}c_kN_k(T)<\infty
 \qquad\text{almost surely}.
\]
Hence $A(t)=\sum_kc_kN_k(t)$ is a continuous adapted process, the series converging uniformly on $[0,T]$ almost surely. Define
\[
 \rho_j:=\inf\{t\in[0,T]:A(t)\ge j\}\wedge T,
 \qquad \inf\varnothing:=T.
\]
Then $\rho_j$ increases and $\PP\{\rho_j=T\}\to1$. Continuity gives
$N_k(\rho_j)\le j/c_k$ for every $j,k$. The usual H\"older embeddings on the bounded cylinder now give a deterministic bound at every exponent $\beta<\alpha$ on each $G\times[0,\rho_j]$. Thus
\[
 u\in\mathcal C_{l.b.}^{2+\alpha-,\alpha/2-}(Q_T).
\]
Finally, any two solutions in this class agree on the intersections of the coefficient- and solution-localizing stochastic intervals by the stopped $W^{2,2}$ uniqueness above. These intervals exhaust $[0,T]$ almost surely, which proves uniqueness and completes the proof.

%%%%%%%%%%%%%%%%%%%%%%%%%%%%%%%%%%%%%%%%%%%%%%%%%%%%%%%
%%% Acknowledgements
%%%%%%%%%%%%%%%%%%%%%%%%%%%%%%%%%%%%%%%%%%%%%%%%%%%%%%%
\section*{Acknowledgments}

The author was supported by the National Natural Science Foundation of China (42450269, 12222103), and LMNS at Fudan University. 
The author also thanks Shanghai Institute for Mathematics and Interdisciplinary Sciences (SIMIS), China for their financial support. 
This research was partly funded by SIMIS, China under grant number SIMIS-ID-2024-WE.

%%%%%%%%%%%%%%%%%%%%%%%%%%%%%%%%%%%%%%%%%%%%%%%%%%%%%%%
%%% Reference section
%%%%%%%%%%%%%%%%%%%%%%%%%%%%%%%%%%%%%%%%%%%%%%%%%%%%%%%
\providecommand{\bysame}{\leavevmode\hbox to3em{\hrulefill}\thinspace}

\end{document}